\documentclass{amsart}
\usepackage{amssymb, amsmath, amsfonts, amsthm, graphics, enumerate, mathrsfs, mathtools,tikz-cd,soul}
\usetikzlibrary{positioning,arrows,scopes}
\usepackage{geometry}       
\definecolor{cadmiumgreen}{rgb}{0.0, 0.42, 0.24}
\usepackage[
colorlinks, citecolor=cadmiumgreen,
pagebackref,
pdfauthor={David Harry Richman}, 
pdfstartview ={FitV},
]{hyperref}
\hypersetup{
pdftitle={The tropical Manin--Mumford conjecture},
}

\newtheorem{thm}{Theorem}[section]
\newtheorem*{thm*}{Theorem}
\newtheorem{theorem}[thm]{Theorem}
\newtheorem{obs}[thm]{Observation}
\newtheorem{prop}[thm]{Proposition}
\newtheorem{lem}[thm]{Lemma}
\newtheorem{cor}[thm]{Corollary}

\theoremstyle{definition}

\newtheorem{dfn}[thm]{Definition}
\newtheorem{thmdfn}[thm]{Theorem - Definition}
\newtheorem{definition}[thm]{Definition}
\newtheorem{eg}[thm]{Example}
\newtheorem{rmk}[thm]{Remark}

\newcommand{\RR}{\mathbb{R}}
\newcommand{\ZZ}{\mathbb{Z}}
\newcommand{\QQ}{\mathbb{Q}}

\newcommand{\cC}{\mathcal{C}}
\newcommand{\cT}{\mathcal{T}}

\DeclareMathOperator{\val}{val}

\DeclareMathOperator{\Div}{Div}
\DeclareMathOperator{\Eff}{Eff}

\DeclareMathOperator{\Pic}{Pic}
\DeclareMathOperator{\Jac}{Jac}

\DeclareMathOperator{\PL}{PL}
\DeclareMathOperator{\Divisor}{\Delta}
\DeclareMathOperator{\zeros}{\Divisor^{+}}
\DeclareMathOperator{\poles}{\Divisor^{--}}

\DeclareMathOperator{\im}{im}

\DeclareMathOperator{\stable}{st}
\DeclareMathOperator{\sgn}{sgn}

\newcommand{\Jactors}[1]{\Jac({#1})_\mathrm{tors}}
\newcommand{\Jtors}[1]{\Jac({#1})_{\rm tors}}
\newcommand{\Abelj}[1]{\iota_{#1}}
\newcommand{\Abeljh}[2]{\iota_{#1}^{(#2)}}
\newcommand{\gamind}{\gamma^{\rm ind}}
\newcommand{\tpacket}[1]{ \{ [{#1}] \}_\mathrm{tors}}

\newcommand{\trees}{\cT}
\newcommand{\cycles}{\cC}

\newcommand{\potent}[2]{j_{#1}^{#2}}

\begin{document}
\title{The tropical Manin--Mumford conjecture}
\author{David Harry Richman}
\date{\today.} 

\begin{abstract}
In analogy with the Manin--Mumford conjecture for algebraic curves,
one may ask how a metric graph under the Abel--Jacobi embedding  intersects torsion points of its Jacobian.
We show that the number of torsion points is finite
for metric graphs of genus $g\geq 2$
which are biconnected and have edge lengths which are ``sufficiently irrational'' in a precise sense.
Under these assumptions, 
the number of torsion points
is bounded by $3g-3$.
Next we study bounds on the number of torsion points in the image of higher-degree Abel--Jacobi embeddings,
which send $d$-tuples of points to the Jacobian.
This motivates the definition of the ``independent girth'' of a graph,
a number which is a sharp upper bound for $d$ such that the higher-degree  Manin--Mumford property holds.
\end{abstract}


\maketitle

\setcounter{tocdepth}{1}
\tableofcontents

\section{Introduction}

Suppose $X$ is a smooth algebraic curve,
let  $\Jac(X)$ denote the Jacobian of $X$,
and let $\Abelj{q} : X \to \Jac(X)$ denote
the Abel--Jacobi map with basepoint $q\in X$.
Raynaud's theorem~\cite{Ray},
formerly the Manin--Mumford conjecture, 
states that if $X$ has genus $g\geq 2$ 
then the image $\Abelj{q}(X)$ 
intersects only finitely many torsion points of $\Jac(X)$. 


\subsection{Statement of results}
The setup above makes sense when the algebraic curve is replaced with a compact, connected metric graph.
We assume throughout this paper that all metric graphs are compact.
Say a metric graph $\Gamma$ 
satisfies the {\em Manin--Mumford condition}
if the image of the Abel--Jacobi map 
$\Abelj{q} : \Gamma \to \Jac(\Gamma)$ 
intersects only finitely many torsion points of $\Jac(\Gamma)$, 
for every choice of basepoint $q\in \Gamma$.

The {\em genus} of a metric graph is the dimension of its first homology group $H_1(\Gamma, \RR)$.
As with algebraic curves, the interesting case to consider 
is when $\Gamma$ has genus $g\geq 2$.
What happens in smaller genus?
When $\Gamma$ has genus zero, the 
Manin--Mumford condition holds vacuously
since $\Jac(\Gamma)$ is a single point.
When $\Gamma$ has genus one, the Abel--Jacobi map is surjective
and $\Jactors{\Gamma}$ is infinite,
so the Manin--Mumford condition does not hold.

\begin{theorem}[Conditional uniform Manin--Mumford bound]
\label{thm:intro-conditional-mm}
Let $\Gamma$ be a  connected metric graph of genus $g\geq 2$.
If the set of torsion points
$ \Abelj{q}(\Gamma)\cap \Jtors{\Gamma} $
is finite,
then we have the uniform bound
\[ \#(\Abelj{q}(\Gamma)\cap \Jtors{\Gamma}) \leq 3g-3. \qedhere\]
\end{theorem}

Unlike the case of algebraic curves,
for a metric graph the genus condition  $g\geq 2$
is not sufficient to imply that $\#(\Abelj{q}(\Gamma)\cap \Jactors{\Gamma})$
is finite.
On a graph with unit edge lengths 
the degree-zero divisor classes 
supported on vertices form a finite abelian group, 
known as the {\em critical group} of the graph (see Section~\ref{subsec:critical-group}). 
In particular, vertex-supported divisor classes are always torsion.
\begin{obs}
\label{obs:intro-mm-rational}
Suppose $\Gamma = (G,\ell)$ is a metric graph of genus $g\geq 2$
whose edge lengths are all rational,
i.e. $\ell(e) \in \QQ_{>0}$ for all $e \in E(G)$.
Then $\Gamma$ does not satisfy the Manin--Mumford condition.
\end{obs}


We say that a property holds for a {\em very general} point of some (real) parameter space 
if it holds outside of a countable collection of codimension-$1$ families.
Recall that a graph $G$ is
{\em biconnected} 
(or {\em two-connected})
if $G$ is connected after deleting any vertex.

\begin{thm}[Generic tropical Manin--Mumford]
\label{thm:intro-manin-mumford}
Let $G$ be a finite connected graph of genus $g\geq 2$.
If $G$ is biconnected,
then for a very general choice of edge lengths $\ell : E(G) \to \RR_{>0}$,
the metric graph
$\Gamma = (G,\ell)$ 
satisfies the Manin--Mumford condition.
\end{thm}



Say a metric graph $\Gamma$ 
satisfies the {\em (generalized) Manin--Mumford condition
in degree $d$}
if the image of the degree $d$ Abel--Jacobi map
\begin{align*}
\Abeljh{D}{d} : \Gamma^{d} &\to \Jac(\Gamma) \\
(p_1,\ldots,p_d) &\mapsto [p_1 + \cdots + p_d - D]
\end{align*}
intersects only finitely many torsion points of $\Jac(\Gamma)$, 
for every choice of effective, degree $d$ divisor class $[D]$. 

When $d = 1$ this is 
the Manin--Mumford condition on $\Gamma$.
If the generalized Manin--Mumford condition holds in degree $d$, then it also holds in degree $d'$ for any $1\leq d' \leq d$.
When $d \geq g $ and $ g \geq 1$, the generalized Manin--Mumford condition cannot hold, 
since the higher Abel--Jacobi map will be surjective
and $\Jtors{\Gamma}$ is infinite.

The following result is a higher-degree analogue of Theorem~\ref{thm:intro-conditional-mm}.
\begin{thm}[Conditional uniform Manin--Mumford bound in higher degree]
\label{thm:intro-conditional-higher}
Let $\Gamma$ be a  connected metric graph of genus $g\geq 1$.
If $\Gamma$ satisfies the Manin--Mumford condition in degree $d$,
 then 
\[ \#(\Abeljh{D}{d}(\Gamma)\cap \Jtors{\Gamma}) \leq  \binom{3g-3}{d} . \qedhere\]
\end{thm}

Recall that the {\em girth} of a graph is the minimal number of edges in a cycle;
this number provides a constraint on $d$ such that the degree $d$ Manin--Mumford condition holds.
Note that if $G$ has genus $\geq 2$ and girth $1$, i.e. $G$ has a loop edge,
then $G$ cannot be biconnected.
\begin{obs}
\label{obs:intro-mm-girth}
Let $G$ be a finite connected graph
with girth $\gamma$.
Then for any choice of edge lengths
the metric graph $\Gamma = (G,\ell)$ does not satisfy
the generalized Manin--Mumford condition in degree
$d\geq \gamma$.
\end{obs}
The bound in Observation~\ref{obs:intro-mm-girth} is not sharp, and can be strengthened.
We define the {\em independent girth} $\gamind$
of a graph as
\[
	\gamind(G) = \min_{C} \Big( \# E(C) + 1 - h_0(G\backslash E(C)) \Big)
\]
where the minimum is taken over all  cycles $C$ in $G$,
and $h_0$ denotes the number of connected components.
Since the (usual) girth is by definition
${\gamma(G) = \min_C (\# E(C)) }$
and $h_0 \geq 1$,
we have the inequality $\gamind \leq \gamma$.
In contrast with girth, the independent girth is invariant under subdivision of edges, so it is well-defined for a metric graph.
The independent girth may be expressed in terms of the cographic matroid of $G$, see Section~\ref{subsec:matroids}.

Say a graph $G$ is {\em Manin--Mumford finite in degree $d$}
if for very general edge lengths $\ell$,
the metric graph $(G,\ell)$ satisfies the degree $d$ Manin--Mumford condition.
\begin{thm}[Generic tropical Manin--Mumford in higher degree]
\label{thm:intro-mm-higher-degree}
Let $G$ be a finite connected graph of genus $g\geq 1$
with independent girth $\gamind$.
Then $G$ is Manin--Mumford finite in degree $d$
if and only if ${ 1 \leq d < \gamind }$.
%
\end{thm}

\begin{rmk}
In Theorem~\ref{thm:intro-manin-mumford}, the conclusion holds for a larger family of graphs than just biconnected graphs: it suffices that each biconnected component has genus $2$ or more.
This can be seen as a special case of Theorem~\ref{thm:intro-mm-higher-degree} since having independent girth $\gamind = 1$ is equivalent to having a biconnected component with genus one.
\end{rmk}

The main results in this paper are proved using the following ingredients: (1) we show that torsion equivalence between divisors can be reduced to a condition on slopes of (sums of) potential functions $j^{y}_{z}$,
and (2) Kirchhoff's Theorem (Theorem~\ref{thm:kirchhoff}) expresses these slopes as rational functions of the edge lengths.

\subsection{Previous work}

Faltings's theorem (previously Mordell's conjecture) states that 
a smooth curve of genus $g \geq 2$ 
has finitely many rational points,
i.e. points whose coordinates are all rational numbers.
By analogy with Mordell's conjecture,
Manin and  Mumford conjectured that
a smooth algebraic curve of genus $g \geq 2$
has finitely many torsion points.
The Manin--Mumford conjecture was proved by Raynaud \cite{Ray},
which inspired several generalizations concerning torsion points
in abelian varieties.
Soon after Raynaud's work different proofs were given by
Coleman, Hindry, Buium, and others.
For an informative survey of this history, see Tzermias~\cite{Tze} and Poonen~\cite{Poo}.
For an introduction to this area of research that also provides background on abelian varieties, see Hindry~\cite{Hin}.

After Raynaud's work, it was still unknown whether the number of torsion points on a genus $g$ curve could be bounded as a function of $g$;
this became known as the 
{\em uniform Manin--Mumford conjecture}.
Baker and Poonen~\cite{BP} extended Raynaud's result quantitatively
by proving strong bounds on the number of torsion points that arise on a given curve as the basepoint for the Abel--Jacobi map varies.
In particular, they showed that a curve $X$ of genus $g\geq 2$
has finitely many choices of basepoint $q$ so that
$\#(\Abelj{q}(X) \cap \Jtors{X})$
has size greater than two.
Katz, Rabinoff, and Zureick-Brown 
\cite{KRZB}
used tropical methods to prove a uniform bound on the number of torsion points on an algebraic curve of fixed genus,
which satisfy an additional technical constraint on the reduction type.
Very recently, K\"uhne~\cite{Kuh} (in characteristic zero) and Looper, Silverman, and Wilms~\cite{LSW}  (in positive characteristic) have found proofs of uniform bounds on the number of torsion points on an algebraic curve.
K\"uhne's result combined with recent work of Dimitrov, Gao, and Habegger~\cite{DGH} is sufficient to prove the uniform Mordell--Lang conjecture;
in this variant of uniform Manin--Mumford, roughly, ``torsion point'' is replaced with ``point in an arbitrary finite-rank subgroup,'' where the bound depends on the rank of the subgroup.
These recent developments are surveyed by Gao~\cite{Gao}.

Regarding the higher-degree Manin--Mumford conjecture,
 Abramovich and Harris \cite{AH} studied the question of when the locus $W_d(X)$ of effective degree $d$ divisor classes on an algebraic curve $X$ contains an abelian subvariety of $\Jac(X)$.
This question was studied further by Debarre and Fahlaoui \cite{DF}.
In~\cite{Hin}, Hindry remarks that ``a geometric description ... of the condition `$W_r$ does not contain a translate of a nonzero abelian subvariety' ... does not seem to be known in general.''
It would be interesting to investigate 
whether 
the results presented here on the higher-degree tropical Manin--Mumford conjecture can lead to progress on the analogue for algebraic curves.
For an overview of the basic theory of tropical geometry and its relation to arithmetic geometry, including recent applications of this framework, see the survey of Baker and Jensen~\cite{BJ}, in particular Section 11.2.

\subsection{Notation}
Here we collect some notation which will be used throughout the paper.

\begin{tabular}{ll}
$\Gamma$ & a compact, connected metric graph \\

$(G,\ell)$ &
a combinatorial model for a metric graph,
where  
$\ell : E(G) \to \RR_{>0}$ \\
& is a length function on edges of $G$ \\

$G$ &
a finite, connected combinatorial graph
(loops and parallel edges allowed) \\

$E(G)$ &
edge set of $G$ \\

$V(G)$ &
vertex set of $G$ \\

$\trees(G)$ & set of spanning trees of $G$ \\

$\cycles(G)$ & set of cycles of $G$ \\

$D$ & a divisor on a metric graph \\

$\PL_\RR(\Gamma)$ & set of piecewise linear functions on $\Gamma$ \\

$\PL_\ZZ(\Gamma)$ & set of piecewise $\ZZ$-linear functions on $\Gamma$ \\

$\Divisor(f)$ & the principal divisor associated to a piecewise ($\ZZ$-)linear function $f$ \\

$\Div(\Gamma)$ & divisors (with $\ZZ$-coefficients) on $\Gamma$ \\


$\Div^d(\Gamma)$ & divisors of degree $d$ on $\Gamma$ \\

$[D]$ &
the set of divisors linearly equivalent to $D$ \\
& (i.e. the linear equivalence class of the divisor $D$) \\

$|D|$ &
the set of effective divisors linearly equivalent to $D$ \\



$\Pic^d(\Gamma)$ & divisor classes of degree $d$ on $\Gamma$ \\

$\Eff^d(\Gamma)$ & effective divisor classes of degree $d$ on $\Gamma$ \\

$\Jac(\Gamma)$ &
the Jacobian of $\Gamma$,
$\Jac(\Gamma) = \Div^0(\Gamma) / \Divisor(\PL_\ZZ(\Gamma)) $ \\

\end{tabular}

\section{Graphs and matroids}
Throughout this paper, 
we use ``graph'' to mean ``oriented graph,''
i.e. each graph is equipped with an orientation on its edges. 
The orientation is auxiliary in the sense that the results will not depend on the orientation chosen.

A {\em graph} $G = (V,E)$
consists of a finite set of vertices $V = V(G)$ and a finite set of edges $E = E(G)$,
equipped with two maps $\mathrm{head}: E\to V$ and $\mathrm{tail} : E \to V$,
which we abbreviate by $e^+ = \mathrm{head}(e)$ and $e^- = \mathrm{tail}(e)$.
We say an edge $e$ lies between vertices $v,w$ if $(e^+,e^-) = (v,w)$ or $(e^+,e^-) = (w,v)$.
We allow loops (i.e. $e$ such that $e^+ = e^-$) and multiple edges (i.e. distinct edges can share the same head and tail).
The {\em valence} of a vertex $v \in V$ is 
\[
\val(v) = \#\{ e \in E : e^+ = v\} + \#\{e \in E : e^- = v \}.
\]
Note that loop edge at $v$ contributes $2$ to the valence of $v$.

Given a subset of edges $A \subset E$,
we let $G|A$ denote the subgraph of $G$ whose vertex set is $V(G)$ and whose edge set is $A$.
We let $G \setminus A$ denote the subgraph
whose vertex set is $V(G)$ and whose edge set is $E \setminus A$.

For a connected graph $G$, the {\em genus} is defined as
$g(G) = \#E(G) - \#V(G) + 1$.
A {\em cycle} of $G$ is a subgraph which is homeomorphic to a circle.
A {\em spanning tree} of $G$
is a subgraph $T$ with the same vertex set and whose edge set is a subset of the edges of $G$,
such that $T$ is connected and has no cycles.
The number of edges in any spanning tree of $G$ is 
$\#V(G) - 1 = \#E(G) - g(G)$.

\subsection{Critical group}
\label{subsec:critical-group}
Fix a graph $G = (V,E)$  with $n$ vertices, 
enumerated as $\{v_1, \ldots, v_n\}$.
The {\em Laplacian matrix} of $G$ is the $n\times n$ matrix $L$
whose entries are
\begin{equation}
\label{eq:laplacian}
L_{ij} = \begin{cases}
- \#(\text{edges between $v_i$ and $v_j$}) &\text{if }i\neq j \\
\sum_{k \neq i} \#(\text{edges  from $v_i$ to $v_k$}) &\text{if }i=j.
\end{cases}
\end{equation}
(Some sources use the opposite sign convention for the graph Laplacian.)
The matrix $L$ is symmetric and its rows and columns sum to zero.
The Laplacian defines a linear map $\ZZ^V \to \ZZ^V$,
whose image has rank $n-1$ if $G$ is connected.
Let $\epsilon: \ZZ^V \to \ZZ$ denote the linear map
which sums all coordinates.
The image of the Laplacian matrix $L$ lies in the kernel of $\epsilon$.

The {\em critical group} $\Jac(G)$ of a connected graph $G$
is the abelian group 
defined as 
\[
\Jac(G) = \ker(\epsilon) / \im(L) .
\]
The critical group is finite, and its cardinality is the number of spanning trees of $G$ \cite[Theorem 6.2]{Big}.
For more on the critical group, see \cite{BN,Big}
and the references therein.

\begin{eg}
Let $G$ be the ``ice cream cone'' graph,
shown in Figure~\ref{fig:cone-graph},
which has three vertices 
and four edges.
\begin{figure}[h]
\centering
\begin{tikzpicture}
	\coordinate (U) at (-1,1.5);
	\coordinate (V) at (1,1.5);
	\coordinate (W) at (0,0);
	
	\draw (U) node[circle,fill,scale=0.5] {} -- 
		  (W) node[circle,fill,scale=0.5] {} -- 
		  (V) node[circle,fill,scale=0.5] {} -- cycle;
	\draw[rounded corners=5] (U) -- (-0.5,2) -- (0.5,2) -- (V);
	
	\path (U) node[above left] {$u$};
	\path (V) node[above right] {$v$};
	\path (W) node[right=2pt] {$w$};
\end{tikzpicture}
\caption{Ice cream cone graph.}
\label{fig:cone-graph}
\end{figure}
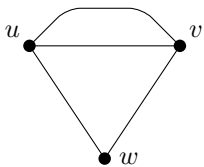
The Laplacian matrix of this graph is 
$
\begin{pmatrix}
3 & -2 & -1 \\
-2 & 3 & -1 \\
-1 & -1 & 2
\end{pmatrix}
$,
and $\Jac(G) \cong \ZZ / 5\ZZ$.
\end{eg}

We now give additional notation for constructing $\Jac(G)$,
which connects the critical group to the Jacobian of a metric graph (see Section~\ref{subsec:int-laplacian}).
Let the {\em divisor group} $\Div(G)$ be the free abelian group on the set of vertices of $G$;
\[
\Div(G) = \{ \sum_{v_i \in V} a_i v_i : a_i\in \ZZ\} ,
\]
and let $\PL_\ZZ(G)$ denote the set of integer-valued functions on vertices of $G$;
\[
\PL_\ZZ(G) = \{ f : V \to \ZZ\} .
\]
(The notation $\PL$ is for {\em piecewise linear};
any function $f : V\to \ZZ$ has a unique linear interpolation along the edges.)
The functions $\PL_\ZZ(G)$ also form a free abelian group.
The set of indicator functions $\{ \mathbf{1}_{v_i} : v_i \in V\}$
forms a basis for $\PL_\ZZ(G)$
where
\[
\mathbf{1}_{v_i}(w) = \begin{cases}
1 &\text{if }w = v_i ,\\
0 &\text{if }w \neq v_i.
\end{cases}
\]
The divisor group by definition has the basis $\{ v_i : v_i\in V\}$.
Let $\Divisor: \PL_\ZZ(G) \to \Div(G)$
be the linear map defined by sending a function $f$ to $- \sum_{v} a_v(f) v$, where $a_v$ is the ``sum of the outgoing slopes'' of $f$ at $v$.
In coordinates,
using the two bases defined above,
the matrix of $\Divisor$ is precisely the Laplacian matrix \eqref{eq:laplacian}.

The degree map $\deg :\Div(G) \to \ZZ$ is defined by $\deg(\sum_i a_i v_i) = \sum_i a_i$.
Let $\Div^0(G)$ denote the kernel of the degree map.
The critical group of $G$ is
\[
\Jac(G) = \Div^0(G) / \Divisor( \PL_\ZZ(G)) .
\]

\begin{rmk}
Both $\Div(G)$ and $\PL_\ZZ(G)$ are isomorphic to $\ZZ^V$, but they are not canonically isomorphic.
$\Div(G)$ is covariant with respect to $V(G_1) \to V(G_2)$, viewed as sets,
while $\PL_\ZZ(G)$ is contravariant.

(Given $f: V(G_1) \to V(G_2)$ on vertex sets there is an induced covariant map $f_*: \Div(G_1) \to \Div(G_2)$ and an induced contravariant map $f^*: \PL_\ZZ(G_2) \to \PL_\ZZ(G_1)$.)
\end{rmk}

\subsection{Metric graphs}
A {\em metric graph} is a compact, connected metric space which comes from 
assigning the path metric to 
a finite, connected graph whose edges are identified with closed intervals
$[0, \ell(e)]$
with positive, real lengths $\ell(e) > 0$.
If the metric graph $\Gamma$
comes from a combinatorial graph $G$ by 
assigning edge lengths $\ell : E(G) \to \RR_{>0}$,
we say $(G,\ell)$ is a {\em combinatorial model} for $\Gamma$
and we write $\Gamma = (G,\ell)$.
A single metric graph generally has many different combinatorial models.

The {\em genus} of a metric graph $\Gamma$ is the dimension of the first homology space,
\[
g(\Gamma) = \dim H_1(\Gamma, \RR).
\]
Given a combinatorial model $\Gamma = (G,\ell)$,
the genus agrees with the underlying graph,
\[
g(\Gamma) = g(G) = \#E(G) - \#V(G) + 1.
\]

The {\em valence} of a point $x$ on a metric graph $\Gamma$,
denoted $\val(x)$,
is the number of connected components in a sufficiently small punctured neighborhood of $x$. 
If $(G,\ell)$ is a model for $\Gamma$, then for any 
$v \in V(G)$ the valence $\val(v)$ as a vertex of $G$ agrees with $\val(v)$
as a point in the metric graph $\Gamma$.

\subsection{Stabilization}
\label{subsec:stabilization}
The notion of stability is useful for our purposes
because questions about Abel--Jacobi maps 
$\Abelj{}: \Gamma \to \Jac(\Gamma)$
may be reduced to 
$\Abelj{}: \Gamma' \to \Jac(\Gamma')$
where $\Gamma'$ is a semistable metric graph. 
(See Section~\ref{subsec:int-laplacian} for discussion of the Abel--Jacobi map.)

A connected 
graph $G$ is {\em stable} if every vertex $v \in V(G)$ has valence at least $3$,
and {\em semistable} if every vertex has $\val(v) \geq 2$.
A metric graph $\Gamma$ is {\em semistable} if every point $x\in \Gamma$ has valence at least $2$.
Equivalently, $\Gamma$ is semistable if it has a model $(G,\ell)$ where $G$ is a  semistable graph.
We say $(G,\ell)$ is a (semi)stable model for $\Gamma$ if $G$ is (semi)stable.

\begin{prop}
A semistable metric graph $\Gamma$ with genus $g\geq 2$ has a unique stable model $(G,\ell)$.
\end{prop}
\begin{proof}
The unique stable model has vertex set
$V(G) = \{ x \in \Gamma : \val(x) \geq 3\}$.
The edges $E(G)$ correspond to connected components of $\Gamma \setminus V(G)$,
which is isometric to a disjoint union of open intervals of finite length.
\end{proof}

\begin{prop}
\label{prop:stable-edge-bound}
Suppose $G$ is a stable graph of genus $g$.
Then the number of edges in $G$ is at most $3g-3$.
\end{prop}
\begin{proof}
Since every vertex has valence at least $3$, we have
\begin{equation*}
  \# V(G) \leq \frac1{3} \sum_{v\in V(G)} \val(v)
 = \frac{2}{3} \cdot \# E(G).
\end{equation*}
By the genus formula $g = \#E(G) - \#V(G) + 1$, this implies
\begin{equation*}
\# E(G) = g - 1+ \#V(G) \leq g - 1 + \frac{2}{3}\cdot \#E(G) 
\end{equation*}
which is equivalent to the desired inequality $\# E(G) \leq 3g-3$.
\end{proof}
It follows from the previous proposition that a stable graph has genus $g\geq 2$.
(We assume stable graphs are connected and nonempty.)
In the following result, recall that a {\em deformation retraction} $r: X \to A$ is a continuous map to a subspace $A \subset X$ that is a homotopy inverse to the inclusion $A \hookrightarrow X$.

\begin{prop}[Metric graph stabilization]
\label{prop:stabilization}
Suppose $\Gamma$ 
has genus $g\geq 1$.
There is a canonical semistable subgraph $\Gamma' \subset \Gamma$
that admits a deformation retraction $r: \Gamma \to \Gamma'$.
\end{prop}
\begin{proof}(Proof sketch)
Call a vertex of valence one a {\em leaf} vertex.
If $\Gamma$ has no leaves, 
then $\Gamma$ is semistable so we take $\Gamma' = \Gamma$. 
Otherwise, choose a leaf vertex 
and contract its incident edge. The result is a metric graph with strictly fewer valence-one vertices, so after a finite number of such operations, we reach a semistable metric graph. We take this result as $\Gamma'$.
At each step, the operation of contracting an edge is a homotopy equivalence.
Thus $\Gamma'$ is homotopy equivalent to $\Gamma$.
We leave the uniqueness of the stabilization $\Gamma'$ as an exercise.
\end{proof}
We call the subgraph $\Gamma'$ of Proposition~\ref{prop:stabilization}
the {\em stabilization} of $\Gamma$,
and denote it as
$\stable(\Gamma)$. 

\begin{eg}
Figure~\ref{fig:stabilization} shows the stabilization $\Gamma'$
of a metric graph $\Gamma$ of genus three.
The deformation retraction $r: \Gamma \to \Gamma'$ 
sends a point of $\Gamma$ to the closest point of $\Gamma'$ in the path metric.
\begin{figure}[h]
\centering
\includegraphics[scale=0.25]{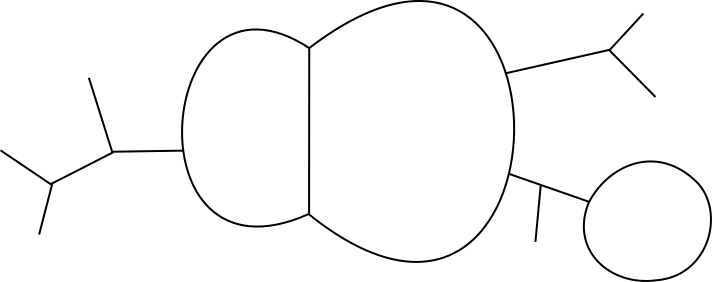}
\qquad\qquad\qquad
\includegraphics[scale=0.25]{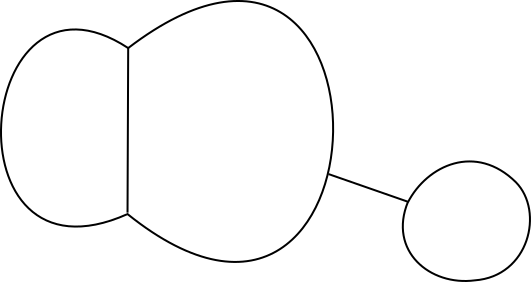}
\caption{A metric graph  (left) and its stabilization  (right).}
\label{fig:stabilization}
\end{figure}
\end{eg}

\subsection{Matroids}
\label{subsec:matroids}

In this section we review the definition of a matroid. 
In particular, we recall the graphic matroid and cographic matroid associated to a connected graph.
Cographic matroids will be useful for understanding the structure 
of the Jacobian of a metric graph (see Section \ref{subsec:int-laplacian}).
For a reference on matroids, see \cite{Oxl} or \cite{Kat}.

A {\em matroid} $M = (E,\mathcal B)$ 
is a finite set $E$ equipped with a nonempty collection  $\mathcal B \subset 2^E$
of subsets of $E$,
called the {\em bases} of the matroid,
satisfying the basis exchange axiom:
for distinct subsets $B_1, B_2 \in \mathcal B$,
there exists some $x \in B_1 \backslash B_2$ and $y \in B_2 \backslash B_1$
such that $(B_1 \backslash x) \cup y \in \mathcal B$.
In other words, 
from $B_1$ we can produce a new basis by exchanging an element of $B_1$ for an element of $B_2$.
An {\em independent set} of a matroid $M = (E,\mathcal B)$ is a subset of $E$ 
which is a subset of some basis.
A {\em cycle} of $M$ is a subset of $E$ which is minimal among those not contained in any basis, under the inclusion relation.
The {\em rank}
of a subset $A\subset E$ is the cardinality of a maximal independent set contained in $A$;
we denote the rank by 
${\rm rk}_M(A)$ or
${\rm rk}(A)$ if the matroid is understood from context.

Given a graph $G = (V,E)$,
the {\em graphic matroid} $M(G)$ is the matroid on the ground set $E = E(G)$
with bases $\mathcal B = \{E(T) : T \text{ is a spanning tree of } G\}$.
An {independent set} in $M(G)$ is a subset of edges which 
span an acyclic subgraph. 
(i.e. $h^1(G|A) = 0$.)
A cycle in $M(G)$ is a cycle in the graph-theoretic sense,
i.e. a subset of edges which span a subgraph homeomorphic to a circle.
The graphic matroid $M(G)$ is also known as the {\em cycle matroid} of $G$.


\begin{eg}
Suppose $G$ is the {\em Wheatstone graph} shown in Figure \ref{fig:wheatstone}.
The bases of $M(G)$ are  $\{abd, abe, acd, ace, ade, bcd, bce, bde\}$.
The cycles are $\{abc, abde, cde\}$.
(Here $abc$ is shorthand for the set $\{a,b,c\}$.)
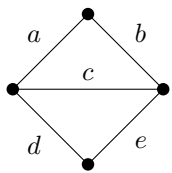
\begin{figure}[h]
\centering
\begin{tikzpicture}
\draw (-1,0) -- (0,1) node[midway,above left] {$a$};
\draw (0,1) -- (1,0) node[midway, above right] {$b$};
\draw (-1,0) -- (1,0) node[midway, above] {$c$};
\draw (-1,0) -- (0,-1) node[midway, below left] {$d$};
\draw (0,-1) -- (1,0) node[midway, below right] {$e$};

\node at (-1,0) [circle,fill,scale=0.5] {};
\node at (1,0) [circle,fill,scale=0.5] {};
\node at (0,-1) [circle,fill,scale=0.5] {};
\node at (0,1) [circle,fill,scale=0.5] {};
\end{tikzpicture}
\caption{Wheatstone graph.}
\label{fig:wheatstone}
\end{figure}
\end{eg}

Given a graph $G = (V,E)$,
the {\em cographic matroid} $M^\perp(G)$ 
is the matroid on the ground set $E = E(G)$
whose bases are complements of spanning trees of $G$.
An {independent set} in $M^\perp(G)$ is a set of edges whose removal does not 
disconnect $G$
(i.e. a set $A \subset E$ such that $G \backslash A$ is connected, 
equivalently 
$h^0(G\backslash A) = 1$).
{An edge set $A\subset E(G)$ is called a {\em cut} of $G$ if $G\backslash A$ is disconnected.}
A cycle in $M^\perp(G)$ is a minimal set of edges $A$ such that 
$h^0(G\backslash A) = 2$;
this is called a {\em simple cut} or a {\em bond} of $G$.
The cographic matroid is also known as the {\em cocycle matroid} or {\em bond matroid}
of $G$.
For more on cographic matroids, see \cite[Chapter 2.3]{Oxl}.

Note: when discussing the graphic or cographic matroid of a graph $G$,
we always use ``cycle of $G$'' to refer to a cycle in the {\em graphic} matroid sense.

\begin{eg}
Suppose $G$ is the Wheatstone graph, shown in Figure \ref{fig:wheatstone}.
The bases of the cographic matroid $M^\perp(G)$ are  
$\{ac, ad, ae, bc, bd, be, cd, ce\}$.
The cycles of $M^\perp(G)$ are 
$\{ab, acd, ace, bcd, bce, de\}$.
\end{eg}

\subsection{Girth and independent girth}

Recall that a {\em cycle} of a graph is a subgraph homeomorphic to a circle.
The {\em girth} $\gamma = \gamma(G)$ of a graph is the minimal number of edges of a cycle.
In other words,
if $\mathcal C(G)$ denotes the set of cycles of $G$, then
\begin{equation}
\label{eq:girth}
\gamma(G) = \min_{C \in \mathcal C(G)} \{ \# E(C) \}.
\end{equation}
We introduce a modification of girth, which naturally extends to an integer-valued invariant on metric graphs.
\begin{definition}
The {\em independent girth} $\gamind$
of a connected graph is defined as
\begin{equation}
\label{eq:ind-girth-rk}
\gamind(G) = \min_{C \in \mathcal C(G)} \{\, {\rm rk}^\perp(E(C)) \,\}
\end{equation}
where ${\rm rk}^\perp$ is the rank function of the cographic matroid $M^\perp(G)$.
(See Section~\ref{subsec:matroids} for discussion of cographic matroids.)
If $G$ has genus zero, we let $\gamind(G) = \gamma(G) = +\infty$.

Equivalently,
\[ \gamind(G) = \min_{C\in \mathcal C(G)} \{\, \# E(C) + 1 - h_0(G\backslash E(C)) \,\} \]
where 
$G\backslash E(C)$ denotes the subgraph obtained by deleting the interior of each edge in $C$,
and $h_0$ denotes the number of connected components of a topological space.
\end{definition}

As a first example, the Wheatstone graph in Figure~\ref{fig:wheatstone} has independent girth $\gamind(G) = 2$.
Indeed, every cycle $C$ in the graph has cographic rank $2$.
In contrast, the shortest cycle has $3$ edges so $\gamma(G) = 3$.

\begin{prop}
\label{prop:gamind}
\begin{enumerate}[(a)]
\item 
For any graph $G$, we have $\gamind(G) \leq \gamma(G)$.

\item 
If 
$(G,\ell)$ and $(G',\ell')$
are combinatorial models for the same metric graph $\Gamma$,
then $\gamind(G) = \gamind(G')$.
\qedhere
\end{enumerate}
\end{prop}
\begin{proof}
(a) 
The rank function (of any matroid) satisfies ${\rm rk}^\perp(A) \leq \# A$. 
The claim follows from comparing definitions \eqref{eq:girth} and \eqref{eq:ind-girth-rk}.


(b)
Since any two combinatorial models of $\Gamma$ have a common refinement by edge subdivisions, it suffices to show that independent girth does not change under subdivision of edges.
Suppose $G'$ is obtained from $G$ by applying edge subdivision.
The cycle sets $\cC(G)$ and $\cC(G')$ are naturally in bijection,
and this bijection preserves cographic rank, since edge subdivision does not change the cographic matroid after simplification.
\end{proof}
Proposition~\ref{prop:gamind}(b)
implies that $\gamind$ is a well-defined invariant for a metric graph.
\begin{dfn}
\label{dfn:gamind-metric}
Given a metric graph $\Gamma$, let
\[
\gamind(\Gamma) := \gamind(G)
\quad
\text{for any choice of model }\Gamma = (G,\ell). \qedhere
\]
\end{dfn}
\noindent
Note that $\gamind$ is also invariant under stabilization (see Section \ref{subsec:stabilization}), 
i.e. the stabilization $\stable(\Gamma)$ satisfies
\begin{equation*}
\gamind(\Gamma) = \gamind(\stable( \Gamma)) .
\end{equation*}

Girth weakly increases under edge deletion, i.e. 
$ \gamma(G\backslash e) \geq \gamma(G)$ for any edge $e$,
since $\mathcal C(G\backslash e) \subset \mathcal C(G)$.
The following examples demonstrate that edge deletion
can increase or decrease $\gamind(G \backslash e)$, relative to $\gamind(G)$.

\begin{eg}
Consider the graphs in  Figure~\ref{fig:ind-girth-ex1}.
The graph on the left has seven simple cycles;
their lengths are
$\{ 4,4,4,6,6,6,6\}$,
and their ranks in the cographic matroid are all $3$.
For this graph, 
$\gamma = 4$ and $\gamind = 3$.
After deleting a central edge, the resulting graph on the right has 
three simple cycles with lengths $\{4, 6, 6\}$
and cographic rank $2$;
hence $\gamma = 4$ and $\gamind = 2$.
\begin{figure}[h]
\centering
		\begin{tikzpicture}[xslant=-0.6,xscale=1.2]
			\coordinate (A) at (0,0);
			\coordinate (B) at (1,0);
			\coordinate (C) at (2,1);
			\coordinate (D) at (2,2);
			\coordinate (E) at (1,2);
			\coordinate (F) at (0,1);
			\coordinate (O) at (1,1);
		
			\draw (A) -- (B) -- (C) -- (D) -- (E) -- (F) -- cycle;
			\foreach \c in {B,D,F} {
				\draw (O) -- (\c);
			}
		
			\foreach \c in {A,B,C,D,E,F,O} {
				\fill[xslant=0.5] (\c) circle (1.8pt);
			}
		\end{tikzpicture}
\qquad\qquad\qquad\qquad
		\begin{tikzpicture}[xslant=-0.6,xscale=1.2]
			\coordinate (A) at (0,0);
			\coordinate (B) at (1,0);
			\coordinate (C) at (2,1);
			\coordinate (D) at (2,2);
			\coordinate (E) at (1,2);
			\coordinate (F) at (0,1);
			\coordinate (O) at (1,1);
		
			\draw (A) -- (B) -- (C) -- (D) -- (E) -- (F) -- cycle;
			\foreach \c in {B,D} {
				\draw (O) -- (\c);
			}
		
			\foreach \c in {A,B,C,D,E,F,O} {
				\fill[xslant=0.5] (\c) circle (1.8pt);
			}
		\end{tikzpicture}
\caption{Graphs with independent girth $3$ (left) and independent girth $2$ (right).}
\label{fig:ind-girth-ex1}
\end{figure}
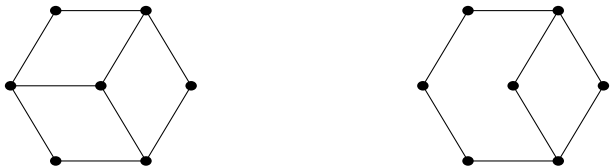
\end{eg}

\begin{eg}
Consider the graph in Figure~\ref{fig:ind-girth-ex2}.
This graph has $\gamma = 4$ and $\gamind = 3$, with the minimum achieved on the $4$-cycle in the middle.
After deleting one of the horizontal edges in the middle cycle,
the resulting graph has $\gamma = 4$ and $\gamind = 4$.
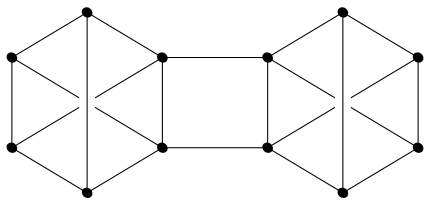
\begin{figure}[h]
\centering
				\begin{tikzpicture}[yslant=-0.6,yscale=1.2]
				\newcommand*{\defcoords}{
					\coordinate (A) at (0,0);
					\coordinate (B) at (1,0);
					\coordinate (C) at (2,1);
					\coordinate (D) at (2,2);
					\coordinate (E) at (1,2);
					\coordinate (F) at (0,1);
					\coordinate (O) at (1,1);
				}
				\defcoords
				
					\draw (A) -- (D);
					\draw (C) -- (F);
					\draw[line width=6pt, color=white] (B) -- (E);
					\draw (B) -- (E);
					\draw (A) -- (B) -- (C) -- (D) -- (E) -- (F) -- cycle;
					\foreach \c in {A,B,C,D,E,F} {
						\fill[xslant=0.5] (\c) circle (1.8pt);
					}
					
					\draw (C) -- +(1.4,0.7);
					\draw (D) -- +(1.4,0.7);

					\begin{scope}[shift={(3.4,1.7)}]
					\defcoords
					\draw (A) -- (D);
					\draw (C) -- (F);
					\draw[line width=6pt, color=white] (B) -- (E);
					\draw (B) -- (E);
					\draw (A) -- (B) -- (C) -- (D) -- (E) -- (F) -- cycle;
				
					\foreach \c in {A,B,C,D,E,F} {
						\fill[xslant=0.5] (\c) circle (1.8pt);
					}
					\end{scope}
				\end{tikzpicture}
\caption{Graph with girth $4$ and independent girth $3$.}
\label{fig:ind-girth-ex2}
\end{figure}
\end{eg}

The following two results give upper bounds on girth and independent girth, in terms of the genus $g$.
To bound the girth (Theorem~\ref{thm:girth-genus-bound}), it is necessary to restrict ourselves to stable graphs---otherwise, the family of $n$-cycles has unbounded girth but fixed genus.
For independent girth, on the other hand, the bound holds for all graphs, and all metric graphs, without needing the (semi)stable assumption.

\begin{thm}
\label{thm:girth-genus-bound}
There exists a constant positive $M$ such that
for any stable graph $G$ of genus $g \geq 2$,
the girth $\gamma = \gamma(G)$
satisfies
$ \gamma < M \log g $.
\end{thm}
\begin{proof}
Recall that the {girth} $\gamma$ of a graph $G$ is 
the minimal number of edges of a cycle in $G$.
Let  $v$ be a vertex in $ V(G)$.
Let $N_r(v)$ denote the neighborhood of radius $r$ around $v$, in the graph $G$.
For any radius $r < \frac12 \gamma$,
the neighborhood $N_r(v)$ is a tree 
(i.e. $N_r(v)$ is connected and acyclic).

Recall that $G$ is {stable}
if every vertex has valence at least $ 3$.
Since $G$ is stable, we may calculate a simple lower bound for the number of edges in $N_r(v)$.
Namely, 
\begin{equation*}
\# E(N_r(v)) \geq 3 + 6 + \cdots + 3\cdot 2^{r-1} = 3(2^r - 1) .
\end{equation*}
This quantity is clearly a lower bound for the total number of edges $\# E(G)$.
Moreover, by Proposition~\ref{prop:stable-edge-bound} we have
$\# E(G) \leq 3g-3 $.
Thus
\begin{equation*}
3(2^r - 1) \leq \# E(G) \leq 3g-3,
\end{equation*}
which implies $2^r \leq g$
for any integer $r < \frac12 \gamma $.
%
Hence
\[ 2^{\gamma/2 - 1} < g ,
\qquad\text{equivalently}\qquad
\gamma < {2} \log_2 g  + 2 .\]
By the assumption $g \geq 2$, this bound implies
$ \gamma < {4}  \log_2 g $,
as desired.
\end{proof}

\begin{cor}
\label{cor:metric-girth-bound}
There exists a constant $M$ such that
for any metric graph $\Gamma$ of genus $g \geq 2$, 
the independent girth $\gamind$ satisfies
$\gamind < M \log g$.
\end{cor}
\begin{proof}
Combine Theorem~\ref{thm:girth-genus-bound}
with Proposition~\ref{prop:gamind}
and Definition~\ref{dfn:gamind-metric}.
\end{proof}

\begin{rmk}
The rank of the cographic matroid $M^\perp(G)$ is equal to the genus of the underlying graph $G$.
Thus it is straightforward from the definition of independent girth that $\gamind(G) \leq g$.
Corollary~\ref{cor:metric-girth-bound} gives a much stronger upper bound on $\gamind$.
\end{rmk}

\section{Divisors on metric graphs}
In this section we recall the theory of divisors and linear equivalence on metric graphs.

On a metric graph $\Gamma$, the {\em divisor group} $\Div(\Gamma)$ is the free abelian group generated by the 
points of $\Gamma$.
We also let $\Div_\RR(\Gamma) = \RR\otimes_\ZZ \Div(\Gamma)$.
In other words,
\begin{align*}
\Div(\Gamma) &= \Big\{ \sum_{x\in \Gamma} a_x x : a_x\in \ZZ,\,
a_x=0 \text{ for almost all } x \Big\} ,
\\
\Div_\RR(\Gamma) &= \Big\{\sum_{x\in \Gamma} a_x x : a_x \in \RR,\, a_x=0 \text{ for almost all }x \Big\}.
\end{align*}
A divisor $D = \sum_{x \in \Gamma} a_x x$ is {\em effective} if $a_x \geq 0$
for every $x$.
The {\em degree} map $\deg : \Div_\RR(\Gamma) \to \RR$
sends $D =\sum_{x \in \Gamma} a_x x$
to $\deg(D) = \sum_{x \in \Gamma} a_x$.

\subsection{Real Laplacian}
\label{subsec:real-laplacian}
A {\em piecewise linear function} on $\Gamma$
is a continuous function $f: \Gamma \to \RR$ which is linear on each edge of some combinatorial model. 
In other words, there exists a model $\Gamma = (G,\ell)$ such that
$
\frac{d}{dt}f(t)
$
is constant
on the interior of each edge $e$ in $E(G)$,
where $t$ is a length-preserving parameter on $e$.
We say a model $(G,\ell)$ is {\em compatible with} a piecewise linear function $f$
if $f$ is linear on each edge of $G$.
We let $\PL_\RR(\Gamma)$ denote the set of all piecewise linear functions on $\Gamma$,
which has the structure of a vector space over $\RR$.

The {\em metric graph Laplacian} $\Divisor$ is the $\RR$-linear map from 
$\PL_\RR(\Gamma)$ to $\Div_\RR(\Gamma)$ defined 
by taking the sum of incoming slopes at each point.
In more detail, for $f \in \PL_\RR(\Gamma)$,
let $(G,\ell)$ be a model for $\Gamma$ which is compatible with $f$,
and let
\begin{equation}
\Divisor(f) = - \sum_{v \in V(G)} a_v v
\qquad\text{where}\qquad
a_v = \sum_{\substack{e \in E(G) \\ e^+ = v \\ \text{or } e^- = v}} f'(t),
\end{equation}
such that in each summand $f'(t) = \frac{d}{dt}f(t)$,
the parameter $t$ is directed away from the vertex $v$.
Equivalently, the coefficient of $v$ in $\Divisor(f)$ is
\begin{align*}
- a_v 
&= 
- \sum_{\substack{e \in E(G) \\ e^+ = v }} \frac{f(e^-)}{\ell(e)}
- \sum_{\substack{e \in E(G) \\ e^- = v}} \frac{f(e^+)}{\ell(e)}
+ \Bigg( \sum_{\substack{e \in E(G) \\ e^+ = v \\ 
\text{or }e^- = v}} \frac{1}{\ell(e)} \Bigg) f(v) .
\end{align*}
Note that for any $f\in \PL_\RR(\Gamma)$, the divisor $\Divisor(f)$ has degree zero.
For any piecewise linear function $f$ there is a unique way to write
$\Divisor(f) = D - E$
where $D$ and $E$ are effective;
we call $\zeros(f) = D$ the {\em divisor of zeros} of $f$
and call $\poles(f) = E$ the {\em divisor of poles} of $f$.

\begin{prop}[Slope-current principle]
\label{prop:slope-current}
Suppose $f \in \PL_\RR(\Gamma)$ 
has zeros $\zeros(f)$ and poles $\poles(f)$ of degree $d\in \RR$.
Then the slope of $f$ is bounded by $d$, i.e.
\[ 
	|f'(x)| \leq d \qquad \text{for any $x \in \Gamma$ where $f$ is linear}. \qedhere
\]
\end{prop}
\begin{proof}
See \cite[Proposition 3.6]{R}.
\end{proof}

\begin{rmk}
The above proposition has a ``physical'' interpretation: 
$f$ gives the voltage in the resistor network $\Gamma$ when subjected to an external current of
$\poles(f)$ units flowing into the network
and $\zeros(f)$ units flowing out.
The slope $|f'(x)|$ is equal to the current flowing through the wire containing $x$, 
which must be no more than the total in-flowing (or out-flowing) current.
\end{rmk}
\noindent The bound in Proposition~\ref{prop:slope-current} is attained only on bridge edges,
and only when all zeros are on one side of the bridge
and all poles are on the other side.

As remarked earlier, the divisor $\Divisor(f)$ has degree zero for any $f\in \PL_\RR(\Gamma)$.
In the converse direction, 
it is natural to ask, 
which degree-zero divisors are equal to $\Divisor(f)$ for some $f$?
The answer is ``all of them.''
In particular, 
taking the divisor $D = y - z$ 
for arbitrary points $y,z \in \Gamma$,
there exists a function $f$ satisfying 
$\Divisor(f) = y - z$.

\begin{thmdfn}
\label{dfn:unit-potential}
For $y, z \in \Gamma$,
the {\em unit potential function} $\potent{z}{y}$ is the unique function in $\PL_\RR(\Gamma)$ satisfying
\[
\Divisor(\potent{z}{y}) = y - z 
\quad \text{and}\quad 
\potent{z}{y}(z) = 0. \qedhere
\]
\end{thmdfn}
\noindent The existence of the function $\potent{z}{y}$ can be shown by the explicit combinatorial formulas due to Kirchhoff (Theorem~\ref{thm:kirchhoff}),
which we discuss in Section~\ref{subsec:kirchhoff}.
The uniqueness of $\potent{z}{y}$ is easier to see;
indeed, the piecewise-linear functions $f$ which have $\Divisor(f) = 0$ are exactly the constant functions.
The observations of the current paragraph show that the Laplacian $ \Divisor: \PL_\RR(\Gamma) \to \Div_\RR(\Gamma)$
fits in an exact sequence
\[
0 \longrightarrow \RR \longrightarrow \PL_\RR(\Gamma) \xrightarrow{\;\Divisor\;} \Div_\RR(\Gamma) 
\longrightarrow \RR \longrightarrow 0,
\]
where the image of  $\RR \to \PL_\RR(\Gamma)$ is the set of constant functions on $\Gamma$,
and $\Div_\RR(\Gamma) \to \RR$ is the degree map.

\subsection{Integer Laplacian and Jacobian}
\label{subsec:int-laplacian}
Recall that $\Div(\Gamma)$ denotes the free abelian group generated by points of $\Gamma$.
A {\em piecewise $\ZZ$-linear function} on $\Gamma$ is a piecewise linear function whose slopes are integers, i.e.
there exists some model $\Gamma = (G,\ell)$ such that
$f'(t) \in \ZZ$ on the interior of each edge of $G$.
We let $\PL_\ZZ(\Gamma)$ denote the set of all piecewise $\ZZ$-linear functions on $\Gamma$.

The Laplacian $\Divisor$, defined in Section \ref{subsec:real-laplacian},
restricts to a map 
$
\Divisor: \PL_\ZZ(\Gamma) \to \Div(\Gamma).
$
Two divisors $D,E$ are {\em linearly equivalent} if there is some $f \in \PL_\ZZ(\Gamma)$ such that
$D = E + \Divisor(f) $.
We let $[D]$ denote the linear equivalence class of a divisor $D$,
i.e.
\[
[D] = \{ E \in \Div(\Gamma) : E = D + \Divisor(f) \text{ for some }f \in \PL_\ZZ(\Gamma) \}
\]

The {\em Picard group} $\Pic(\Gamma)$ is defined as the cokernel of 
$\Divisor : \PL_\ZZ(\Gamma) \to \Div(\Gamma)$.
The elements of $\Pic(\Gamma)$ are linear equivalence classes of divisors on $\Gamma$.
The integer Laplacian map fits in an exact sequence
\begin{equation*}
0 \longrightarrow \RR \longrightarrow \PL_\ZZ(\Gamma) \xrightarrow{\;\Divisor\;} \Div(\Gamma) \longrightarrow \Pic(\Gamma) \longrightarrow 0 .
\end{equation*}
The degree of a divisor class is well-defined,
so we have an induced degree map 
$\Pic(\Gamma) \to \ZZ$.
The {\em Jacobian group} $\Jac(\Gamma)$ is the kernel of this degree map,
so we have a short exact sequence
\[
0 \longrightarrow \Jac(\Gamma) \longrightarrow \Pic(\Gamma) \xrightarrow{\;\mathrm{deg}\;} \ZZ \longrightarrow 0.
\]
This short exact sequence splits, 
$\Pic(\Gamma) \cong \ZZ \times \Jac(\Gamma)$,
but this isomorphism is not canonical.
One way to obtain a splitting is to choose a point $q \in \Gamma$,
and define $\ZZ \to \Pic(\Gamma)$
by sending $n$ to the divisor class $[nq]$.
We also denote $\Jac(\Gamma)$ by $\Pic^0(\Gamma)$,
and use $\Pic^d(\Gamma)$ to denote the divisor classes of degree $d$.

The tropical Abel--Jacobi theorem, due to Mikhalkin and Zharkov {\cite[Theorem 6.2]{MZ}}, 
identifies the structure of $\Jac(\Gamma)$ as a connected topological abelian group.
\begin{thm}[Tropical Abel--Jacobi]
\label{thm:abel-jacobi}
Suppose $\Gamma$ is a metric graph of genus $g$.
Then
\begin{equation*}
\Jac(\Gamma) \cong \RR^g / \ZZ^g . \qedhere
\end{equation*}
\end{thm}

Fix a basepoint $q$ on the metric graph $\Gamma$.
The {\em Abel--Jacobi map}
\begin{equation}
\label{eq:abel-jacobi}
\Abelj{q} : \Gamma \to \Jac(\Gamma)
\end{equation}
sends a point $ x\in \Gamma$ to the divisor class $[x-q]$.
More generally, 
if $D$ is a degree $d$ divisor
we have a higher-dimensional Abel--Jacobi map
\begin{equation*}
\Abelj{D}^{(d)} :  \Gamma^d \to \Jac(\Gamma)
\end{equation*}
which sends a tuple $(x_1,\ldots,x_d)$ to the divisor class
$[x_1 + \cdots + x_d - D]$.

\begin{prop}
\label{prop:jacobian-stabilization}
Let $\stable(\Gamma)$ denote the stabilization of $\Gamma$
(see Section~\ref{subsec:stabilization}).
\hfill
\begin{enumerate}[(a)]
\item 
The deformation retraction $r: \Gamma \to \stable(\Gamma)$
induces an isomorphism
$\Jac(\Gamma) \to \Jac(\stable(\Gamma))$ 
on Jacobians.

\item 
The inclusion $i: \stable(\Gamma) \to \Gamma$
induces an isomorphism
$\Jac(\stable(\Gamma)) \to \Jac(\Gamma)$ 
on Jacobians.
\qedhere
\end{enumerate}
\end{prop}

For a proof and further motivation, see Caporaso~\cite{Cap}.

\subsection{Cellular decomposition of the Jacobian}
In this section we recall how the geometry of the Abel--Jacobi map $\Abelj{q}: \Gamma \to \Jac(\Gamma)$ is related to the cographic matroid $M^\perp(G)$,
where $(G,\ell)$ is a model for $\Gamma$.
We describe cellular decompositions of the subset of effective divisor classes inside $\Pic^k(\Gamma)$ for $k\geq 0$;
each $\Pic^k(\Gamma)$ can be identified with $\Jac(\Gamma)$
by subtracting a fixed degree $k$ divisor class.

A consequence of Mikhalkin and Zharkov's proof \cite{MZ}
of the tropical Abel--Jacobi theorem (Theorem \ref{thm:abel-jacobi})
is that the Abel--Jacobi map 
$\Gamma \to \Jac(\Gamma)$ 
is linear on each edge of $\Gamma$.
The universal cover of $\Jac(\Gamma)$
is naturally identified with $H^1(\Gamma,\RR)$.
The Abel--Jacobi map,
restricted to a single edge $e \subset \Gamma$,
lifts locally to $e \to H^1(\Gamma, \RR)$.
The linear independence of the edge-vectors in the image $\Gamma \to \Jac(\Gamma)$
is exactly recorded by the cographic matroid $M^\perp(G)$,
for any combinatorial model $\Gamma = (G,\ell)$ (see Theorem~\ref{thm:cographic-dim}).

\begin{definition}
\label{dfn:eff-cell}
Let $\Gamma = (G,\ell)$ be a metric graph.
Given edges $e_1,\ldots,e_k \in E(G)$,
let $\Div(e_1,\ldots,e_k) \subset \Div^k(\Gamma)$
denote the set of effective divisors formed by adding together 
one point from each edge $e_i$.
Let $\Eff(e_1,\ldots, e_k)$ denote the corresponding set of effective divisor classes,
\[
\Eff(e_1,\ldots,e_k) = \{[x_1 + \cdots + x_k] : x_i \in e_i \} \subset \Pic^k(\Gamma).
\qedhere
\]
\end{definition}

The following result relates these cells of effective divisor classes with the cographic matroid (see Section~\ref{subsec:matroids}).
\begin{thm}
\label{thm:cographic-dim}
Let $\Gamma = (G,\ell)$ be  a metric graph.
The dimension of $\Eff(e_1,\ldots, e_k)$
is equal to the rank of $\{e_1,\ldots, e_k\}$
in the cographic matroid $M^\perp(G)$.
\end{thm}

\begin{proof}
For each edge $e_i \in E(G)$,
let $v_i \in H^1(\Gamma,\RR)$ denote a vector
parallel to the Abel--Jacobi image of $e_i$ in $\Jac(\Gamma)$.
Then according to Definition 5.1.3 and Lemma 5.1.6 of \cite{CV},
the set of vectors $\{v_i : e_i \in E(G)\}$
form a realization of the cographic matroid $M^\perp(G)$.
This means that the cographic rank of $\{e_1,\ldots,e_k\}$
agrees with the dimension of the linear span of $\{v_1,\ldots,v_k\}$.

The subset $\Eff(e_1,\ldots,e_k) \subset \Pic^k(\Gamma)$
is naturally identified with the Minkowski sum 
of the corresponding vectors $v_1,\ldots,v_k \in H^1(\Gamma,\RR)$,
so the claim follows.
\end{proof}

It will be useful to consider the space of effective divisor classes of fixed degree $d$.
This space is studied by Gross, Shokrieh, and T\'{o}thm\'{e}r\'{e}sz in 
\cite{GST}.

\begin{thm}
\label{thm:abks-deg-d}
Let $\Gamma = (G, \ell)$ be a metric graph of genus $g$.
For any integer $d$ in the range $0\leq d\leq g$,
the space $\Eff^d(\Gamma)$ 
of degree $d$ effective divisor classes 
is covered by the union of
$\Eff(e_1,\ldots,e_d)$
as $\{e_1,\ldots,e_d\}$ varies over independent sets of size $d$ in the cographic matroid $M^\perp(G)$.
\end{thm}
\begin{proof}
The result \cite[Theorem A]{GST} states that every effective divisor class of degree $d$ contains a semibreak divisor.
The definition of semibreak divisor is equivalent to the condition that the divisor is in $\Eff(e_1,\ldots,e_d)$
for a set of edges $\{e_1,\ldots,e_d\}$ which is independent in the cographic matroid of $G$.

Theorem~\ref{thm:cographic-dim} implies that $\Eff(e_1,\ldots,e_d)$ has dimension $d$ for each independent set
$\{e_1,\ldots,e_d\}$,
so it follows that $\Eff^d(\Gamma)$ is a cellular complex of pure dimension $d$.
\end{proof}

%

A particular example is the case $d = g$, which was studied by An--Baker--Kuperberg--Shokrieh~\cite{ABKS}. 
In this case, the number of maximal cells of $\Pic^g(\Gamma)$ is equal to the number of spanning trees of $G$.

\subsection{Kirchhoff formulas}
\label{subsec:kirchhoff}
In this section we review Kirchhoff's formulas \cite{Kir}
for the unit potential functions $\potent{z}{y}$,
which are fundamental solutions to the Laplacian map (see Definition~\ref{dfn:unit-potential}).
Expositions of this material can be found in Bollob{\'a}s~\cite[$\S$II.1]{Bol}
and Grimmet~\cite[$\S$1.2]{Gri}.

\begin{thm}[Kirchhoff]
\label{thm:kirchhoff}
Suppose $\Gamma = (G,\ell)$ is a metric graph
with edge lengths $\ell : E(G) \to \RR_{>0}$.
For vertices $y,z \in V(G)$, 
let $\potent{z}{y}: \Gamma \to \RR$ denote 
the function in $\PL_\RR(\Gamma)$ which satisfies 
$\Divisor(\potent{z}{y}) = y - z$ and $j^y_z(z) = 0$.
Then the following relations hold.
\begin{enumerate}[(a)]
\item 
For any directed edge $\vec{e} = (e^+,e^-)$,
\begin{equation}
\label{eq:current}
\frac{j^y_z(e^+) - j^y_z(e^-)}{\ell(e) } 
= \frac{\sum_{T \in \trees(G)}\sgn(T,y,z,\vec e) w(T)}{\sum_{T \in \trees(G)} w(T)}
\end{equation}
where $\trees(G)$ denotes the spanning trees of $G$,
the {\em weight} $w(T)$ of a spanning tree is defined as
\[
	w(T) = \prod_{e_i \not\in E(T)} \ell(e_i),
\] 
and 
\[ 
	\sgn(T,y,z,\vec e) = \begin{cases}
	+1 & \text{if the path in $T$ from $y$ to $z$ passes through $\vec e$} \\
	-1 & \text{if the path in $T$ from $y$ to $z$ passes through $-\vec e$} \\
	0 & \text{otherwise}.
	\end{cases}
\]

\item 
The total potential drop between $y$ and $z$ is 
\begin{equation}
\label{eq:voltage}
j^y_z(y) - j^y_z(z) = \frac{\sum_{T \in \trees(G_0)} w(T)}{\sum_{T \in \trees(G)} w(T)}
\end{equation}
using the same notation as above, and where
the graph $G_0$ (in the numerator) is the graph obtained from $G$ by identifying vertices $y$ and $z$.
\qedhere
\end{enumerate}
\end{thm}
\begin{proof}
For part (a), see Bollob\'as \cite[Theorem 2, $\S$II.1]{Bol}.
Part (b) follows 
from consideration of 
the graph $G_+$ obtained by adding an auxiliary edge
to $G$ between $y$ and $z$,
and then applying part (a) to $G_+$ with respect to the auxiliary edge.
\end{proof}

The expressions \eqref{eq:current}, \eqref{eq:voltage} 
are both
a ratio of homogeneous polynomials\footnote{moreover, 
polynomials whose nonzero coefficients  are all $\pm 1$}
in the variables $\{ \ell(e_i) : e_i \in E(G)\}$.
In \eqref{eq:current}, the numerator and denominator are homogeneous of degree $g$;
in \eqref{eq:voltage}, the denominator has degree $g$ while the numerator has degree $g+1$.
As a result, 
the expression \eqref{eq:current} 
is invariant under simultaneous rescaling of edge lengths,
while 
the expression \eqref{eq:voltage} 
scales linearly with respect to simultaneously rescaling all edge lengths.

\begin{eg}
\label{eg:kirchhoff1}
Consider the theta graph shown in Figure~\ref{fig:kirchhoff-ex1},
where $a = \ell(e_1)$, $b = \ell(e_2)$, $c = \ell(e_3)$ are edge lengths.
The spanning trees are $\trees(G) = \{e_3, e_2, e_1\}$
which have respective weights
$\{ ab, ac, bc\} $.
The current along edge $e_1$ is 
\[ 
\frac{j^y_z(y) - j^y_z(z)}{a} = \frac{bc}{ab + ac + bc} ,
\]
according to \eqref{eq:current}.
We have
\begin{align*}
\potent{z}{y}(y) - \potent{z}{y}(z) 
&= a \left( \frac{bc}{ab + ac + bc} \right)
= \frac{abc}{ab + ac + bc} 
\end{align*}
in agreement with \eqref{eq:voltage};
the graph $G_0$ consists of three loop edges. 
Note the symmetry in $a,b,c$.
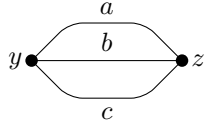
\begin{figure}[h]
\centering
{
	\begin{tikzpicture}
	\draw (-1,0) -- (1,0);
	\draw[rounded corners=5] (-1,0) -- (-0.5,-0.5) -- (0.5,-0.5) -- (1,0);
	\draw[rounded corners=5] (-1,0) -- (-0.5,0.5) -- (0.5,0.5) -- (1,0);
	
	\node at (-1,0) [circle,fill,scale=0.5] {};
	\node at (1,0) [circle,fill,scale=0.5] {};
	
	\path (-1,0) node[left] {$y$};
	\path (1,0) node[right] {$z$};
	
	\path (-0.5,0.5) -- (0.5,0.5) node[midway, above] {$a$};
	\path (-1,0) -- (1,0) node[midway,above] {$b$};
	\path (-0.5,-0.5) -- (0.5,-0.5) node[midway, below] {$c$};

	\end{tikzpicture}
}
\caption{Theta graph  with variable edge lengths.}
\label{fig:kirchhoff-ex1}
\end{figure}
\end{eg}

\section{Torsion points of the Jacobian}
\label{sec:mm-setup}

In this section, we define torsion equivalence on divisor classes.
We recall the definition of a very general subset of $\RR^n$ and give examples.

\subsection{Torsion equivalence}
Given an abelian group $A$,
the {\em torsion subgroup} $A_{\rm tors}$ 
is the set of elements $a\in A$ such that $na = a + \cdots + a =  0$ for some positive integer $n$.
For example, the torsion subgroup of $\RR/\ZZ$ is $\QQ/\ZZ$
and the torsion subgroup of $\RR$ is $\{0\}$.
Recall that  $\Jac(\Gamma)$ 
is the abelian group of degree $0$ divisor classes on $\Gamma$;
we have
\[
\Jtors{\Gamma} = \{[D]  : D\in \Div^0(\Gamma),\quad n[D] = 0\text{ for some }n \in \ZZ_{>0} \}.
\]

We say points $x,y \in \Gamma$ are {\em torsion equivalent}
if there exists a positive integer $n$ such that $n[x-y] = 0$ in $\Jac(\Gamma)$.
If two points $x,y$ represent the same divisor class $[x] = [y]$, then 
$x$ and $y$ are torsion equivalent;
hence this relation descends to a relation on $\Eff^1(\Gamma) = \{ [x] : x \in \Gamma\}$.
It will be convenient for us to consider this relation on 
$\Eff^1(\Gamma)$ rather than  on  $\Gamma$.
\begin{lem}
Torsion equivalence defines an equivalence relation on $\Eff^1(\Gamma)$.
\end{lem}
\begin{proof}
It is clear that torsion equivalence is reflexive and symmetric.
Suppose $n, m$ are positive integers such that 
$n[x-y] = 0$ and $m[y-z] = 0$ in $\Jac(\Gamma)$.
Then $mn[x-z] = mn([x-y] + [y-z]) = 0$. 
This shows that torsion equivalence is transitive.
\end{proof}

It is natural to extend this relation to divisor classes of higher degree:
we say  effective classes $[D], [E] \in \Eff^d(\Gamma)$
are {\em torsion equivalent}
if $n[D - E] = 0$ for some positive integer $n$.
We call an equivalence class under this relation a {\em torsion packet}.

\begin{definition}
\label{dfn:torsion-packet}
Given  $[E] \in \Eff^d(\Gamma)$,
the {\em torsion packet} $\tpacket{E}$ is the set of effective divisor classes
torsion equivalent to $[E]$, i.e.
\[ 
\tpacket{E} = \{ [D] \in \Eff^d(\Gamma) \text{ such that }[D - E] \in \Jtors{\Gamma} \}. \qedhere
\]
\end{definition}

The terminology of torsion packets allows us to restate the Manin--Mumford condition in a basepoint-free manner.
\begin{prop}
\label{prop:mm-torsion-packet}
\hfill
\begin{enumerate}[(a)]
\item 
Given an effective divisor class $[D] \in \Eff^d(\Gamma)$,
there is a canonical bijection
\[
\tpacket{D} \quad\longleftrightarrow\quad \Abeljh{D}{d}(\Gamma^d) \cap \Jtors{\Gamma}
\]
where
$\Abeljh{D}{d} : \Gamma^d \to \Jac(\Gamma)$
is the Abel--Jacobi map \eqref{eq:abel-jacobi}.

\item 
A metric graph $\Gamma$ satisfies the {degree $d$ Manin--Mumford condition} 
if and only if every torsion packet of degree $d$ 
is finite.
\qedhere
\end{enumerate}
\end{prop}
\begin{proof}
For part (a),
we have
\begin{align*}
\tpacket{D} &= \{[E] \in \Eff^d(\Gamma) : [E - D] \in \Jtors{\Gamma} \} \\
&= \{ [x_1 + \cdots + x_d] \in \Eff^d(\Gamma) : [x_1 + \cdots + x_d - D] \in \Jtors{\Gamma} \} \\
&= \{ [x_1 + \cdots + x_d] \in \Eff^d(\Gamma) : \Abeljh{D}{d}(x_1,\ldots,x_d) \in \Jtors{\Gamma} \} \\
&= \Abeljh{D}{d}(\Gamma^d) \cap \Jtors{\Gamma} .
\end{align*}

Part (b) follows directly from (a) and the definitions.
\end{proof}

Recall that the potential function $\potent{y}{x}$ 
is the unique piecewise $\RR$-linear function 
satisfying
\[
\Divisor(\potent{y}{x}) = x - y
\qquad\text{and}\qquad 
\potent{y}{x}(y) = 0.
\]
\begin{lem}
\label{lem:torsion-rational-slope}
Suppose $x,y$ are two points on a metric graph $\Gamma$.
Then $[x-y]$ is torsion in the Jacobian of $\Gamma$ 
if and only if
all slopes of the potential function $j^x_y$ are rational.
\end{lem}
The above lemma is the special case $d = 1$ of the following statement.

\begin{lem}
\label{lem:d-torsion-slope}
Suppose $D = x_1 + \cdots+ x_d$ and $E = y_1 + \cdots + y_d$
are effective divisors of degree $d$ on a metric graph $\Gamma$.
Let $f \in \PL_{\RR}(\Gamma)$ be a function satisfying
$\Divisor(f) = D - E$.
(Up to an additive constant, $f = \sum_{i=1}^d \potent{x_i}{y_i}$.)
\begin{enumerate}[(a)]
\item The divisor class $[D -E ] = 0$ 
if and only if all slopes of $f$ are integers.

\item The divisor class $[D-E]$ is torsion 
if and only if all slopes of $f$ are rational.
\qedhere
\end{enumerate}
\end{lem}
\begin{proof}
For part (a), the ``if'' direction is a restatement of the definition of linear equivalence
(Section~\ref{subsec:int-laplacian}).
The ``only if'' direction follows from fact that for fixed divisors $D$ and $E$,
all solutions to $\Divisor(f) = D - E$
(where $f \in \PL_\RR(\Gamma)$)
have the same slopes,
since they all differ by an additive constant.

Part (b) follows from part (a) by linearity of the Laplacian $\Divisor$;
more precisely,  
$[D-E]$ is torsion of order  $n$
if and only if $[n(D-E)] = [n\Divisor(f)] = [\Divisor(n\cdot f)] =  0$,
if and only if all slopes of $n\cdot f$ lie in $\ZZ$,
if and only if all slopes of $f$ lie in $\frac1{n}\ZZ$.
Conversely if all slopes of $f$ are rational, then there exists an integer $n$
such that all slopes of $f$ lie in $\frac1n \ZZ$,
since a piecewise linear function has finitely many slopes.
\end{proof}

\subsection{Very general subsets}
\label{subsec:very-general}
A {\em very general} subset of $\RR^n$ is one whose complement is 
contained in a countable union of distinguished Zariski-closed sets.
A distinguished Zariski-closed set is the set of zeros of a polynomial function
which is not identically zero.
Given a polynomial $f\in \RR[x_1,\ldots,x_n]$, we denote
\[ 
Z(f) = \{ (a_1,\ldots,a_n) \in \RR^n : f(a) = 0 \}
\quad\text{ and }\quad 
U(f) = \{ (a_1,\ldots,a_n) \in \RR^n : f(a) \neq 0 \} .
\]
In this notation, a very general subset $S\subset \RR^n$ is one which can be expressed as
\[ 
	S \supset \RR^n \setminus \left(\bigcup_{i\in I} Z(f_i) \right) = \bigcap_{i\in I} U(f_i) 
\]
where $I$ is a countable index set and each $f_i$ is a nonzero polynomial.
Note that the zero locus $Z(f)$ 
has Lebesgue measure zero if $f$ is nonzero. 
Thus the complement of a (measurable) very general subset of $\RR^n$ has Lebesgue measure zero.
However, it is still possible that the complement of a very general subset
is dense in $\RR^n$.

If $D \subset \RR^n$ is some parameter space 
with nonempty interior (with respect to the Euclidean topology),
we say that a subset of $D$ is {\em very general} if it has the form
$D \cap S$ for a very general subset $S \subset \RR^n$.
In our applications, the relevant parameter space will be 
the positive orthant $D = (\RR_{>0})^n$.
We say that a property holds for a {\em very general point} of some real parameter space 
if it holds on a very general subset.

\begin{eg}
\label{eg:very-general}
\hfill
\begin{enumerate}[(a)]
\item 
For a fixed nonconstant polynomial $f \in \ZZ[x_1,\ldots,x_n]$,
the set
\begin{equation}
\label{eq:f-irrational}
U(f-\QQ) =  \{(a_1,\ldots,a_n) \in \RR^n : f(a_1,\ldots,a_n) \not\in \QQ\}
\end{equation}
is very general,
since $\{ f - \lambda : \lambda \in \QQ\}$ 
is a countable collection of nonzero polynomials.

\item 
For polynomials $f,g \in \ZZ[x_1,\ldots,x_n]$ with $g\neq 0$ and $f/g$ nonconstant,
the set
\begin{equation}
\label{eq:fg-irrational}
U\Big(\frac{f}{g} - \QQ\Big) = \left\{(a_1,\ldots,a_n) \in \RR^n : \frac{f(a_1,\ldots,a_n)}{g(a_1,\ldots,a_n)} \not\in \QQ \right\} 
\end{equation}
is very general,
since $\{f - \lambda g : \lambda \in \QQ\}$
is a countable collection of nonzero polynomials.

\item 
The set 
\begin{equation}
\label{eq:transcendental}
 U^n_{\rm tr.} = \{(a_1,\ldots,a_n) \in \RR^n : 
 f(a_1,\ldots,a_n) \neq 0 
\text{ for every } 
f \in \ZZ[x_1,\ldots,x_n] \setminus\{0\} 
\} 
\end{equation}
is very general,
since $\ZZ[x_1,\ldots,x_n]$ is a countable set of polynomials.
We call $U^n_{\rm tr.}$
the set of {\em transcendental} points of $\RR^n$.
In particular, $U^1_{\rm tr}$ is the set of transcendental real numbers.
\qedhere
\end{enumerate}
\end{eg}
Note that in the above examples,
the subsets \eqref{eq:f-irrational} and \eqref{eq:fg-irrational}
contain the transcendental points $U^n_{\rm tr.}$. 
Conversely, $U^n_{\rm tr.}$ is the intersection of \eqref{eq:f-irrational}
over all choices of $f$ (resp. the intersection of \eqref{eq:fg-irrational} over all choices of $f$ and $g$).

\begin{rmk}
In the later theorem statements (Theorems~\ref{thm:manin-mumford} and \ref{thm:mm-higher-degree}) 
which concern very general edge lengths,
the condition ``for very general edge lengths'' can be replaced with the more specific condition ``for transcendental edge lengths,''
i.e. when the tuple of lengths $(\ell(e))_{e \in E}$ is in $U^{|E|}_{\rm tr.}$.
The condition can further be relaxed to say ``for edge lengths in a finite intersection of sets of the form \eqref{eq:fg-irrational},''
in which the polynomials $f,g$ come from Kirchhoff's formulas 
(see Theorem~\ref{thm:kirchhoff} in Section~\ref{subsec:kirchhoff}).
\end{rmk}

\section{Manin--Mumford conditions on metric graphs}
Recall that a metric graph $\Gamma$ 
satisfies the {\em Manin--Mumford condition}
if 
$\#( \Abelj{q}(\Gamma) \cap \Jtors{\Gamma}) < \infty$
for every $q\in \Gamma$,
where $\Abelj{q} : \Gamma \to \Jac(\Gamma)$ 
is the Abel--Jacobi map.
(See Section~\ref{subsec:int-laplacian} for definitions of $\Abelj{q}$ and $\Jac(\Gamma)$.)
Equivalently, $\Gamma$ satisfies the Manin--Mumford condition if every degree one torsion packet
$\tpacket{x}$
is finite.

A metric graph $\Gamma$ 
satisfies the {\em degree $d$ Manin--Mumford condition}
if 
$\#( \Abeljh{D}{d}(\Gamma^d) \cap \Jtors{\Gamma}) < \infty$
for every effective degree $d$ divisor class $[D]$. 
Here $\Abeljh{D}{d}: \Gamma^{d} \to \Jac(\Gamma) $ is the degree $d$ Abel--Jacobi map
\begin{align*}
\Abeljh{D}{d} (x_1,\ldots,x_d) &= [x_1 + \cdots + x_d - D].
\end{align*}
Equivalently, 
every degree $d$ torsion packet on $\Gamma$ is finite.

\subsection{Failure of Manin--Mumford condition}
In this section, we consider cases when a metric graph fails to satisfy the Manin--Mumford condition,
in degree one and in higher degree.

\begin{prop}[Observation~\ref{obs:intro-mm-rational}]
\label{prop:mm-rational}
If $\Gamma = (G,\ell)$ is a metric graph of genus $g \geq 1$ whose edge lengths are all rational,
then the Manin--Mumford condition fails to hold.
\end{prop}
\begin{proof}
First, suppose that all edge lengths are integers.
This means that after subdividing edges appropriately, $\Gamma$ has a combinatorial model $(G^{(1)},\textbf{1})$
with all unit edge lengths.
On a graph with unit edge lengths,
the degree-zero divisor classes 
supported on vertices form a finite abelian group, 
known as the {\em critical group} of the graph (see Section~\ref{subsec:critical-group}).
This implies that all vertices of $G^{(1)}$ lie in the same torsion packet, i.e. $\#( \Abelj{q}(\Gamma) \cap \Jtors{\Gamma}) \geq \# V(G^{(1)})$ for any vertex $q$.

Now consider taking the $k$-th subdivision graph $G^{(k)}$ of $G^{(1)}$,
meaning every edge of $G^{(1)}$ is subdivided into $k$ edges of equal length.
The same metric graph $\Gamma$ also has combinatorial model $(G^{(k)}, \frac1{k} \textbf{1})$.
The degree-zero divisor classes supported on vertices of $G^{(k)}$ also form a finite abelian group,
so these new vertices are also in the same torsion packet of $\Gamma$. Thus
\[
\#( \Abelj{q}(\Gamma) \cap \Jtors{\Gamma}) \quad\geq\quad \# V(G^{(k)}) = \# V(G^{(1)}) + (k-1) \# E(G^{(1)}).
\]
Taking $k\to \infty$ shows that  $\Gamma$ has an infinite torsion packet.

Now suppose that $\Gamma$ has rational edge lengths.
The edge set is finite (in any model) so after rescaling all lengths by the common denominator we obtain a metric graph with integer edge lengths.
If $\lambda: \Gamma \to \Gamma'$ denotes a map that rescales all edge lengths of $\Gamma$ by the same positive factor, 
then $\lambda$ induces a bijection
$\Abelj{q}(\Gamma) \cap \Jtors{\Gamma} \xrightarrow{\sim} \Abelj{\lambda(q)}(\Gamma') \cap \Jtors{\Gamma'}$.
Thus the validity of the Manin--Mumford condition is not changed under global rescaling, 
and the previous argument  may be applied.
\end{proof}

Proposition~\ref{prop:mm-rational} also follows from Lemma~\ref{lem:infinite-torsion} below, taking degree $d = 1$.
Recall that given edges $e_i \in E(G)$, $\Eff(e_1,\ldots, e_k)$ denotes
the set of effective divisor classes $[x_1 + \cdots + x_k]$ 
which sum a point $x_i \in e_i$ from each edge
($x_i$ is allowed to be an endpoint of $e_i$).
\begin{lem}
\label{lem:infinite-torsion}
Let $\Gamma = (G,\ell)$ be a metric graph.
If $\Eff(e_1,\ldots, e_d)$ contains distinct divisor classes 
$[D], [E]$
in the same degree $d$ torsion packet,
then the torsion packet 
$\tpacket{D}$
is infinite.
\end{lem}
\begin{proof}
%
Since the cell $\Eff(e_1,\ldots,e_d)$ is convex,
it contains a line segment connecting $[D]$ and $[E]$;
this segment has positive length by the assumption $[D] \neq [E]$.
Moreover, for 
\[
	[F] = (\text{any rational affine combination  of $[D]$ and $[E]$ along this line}), 
\]
we claim that the class $[F - D]$ is torsion.
This claim suffices to guarantee
infinitely many divisor classes $[F]$
in the torsion packet $\tpacket{D}$.

To verify the claim,
suppose $a, b$ are positive integers such that
$(a + b)[F] = a [D] + b [E]$.
Then 
\[
	(a + b)[F - D] = (a[D] + b[E]) - (a + b)[D]
	= b [E - D].
\]
The assumption that $[D], [E]$ are in the same torsion packet means there exists a positive integer $n$ such that $n[E - D] = 0$.
It then follows from the displayed equation that $n(a + b)[F - D] = 0$.
This verifies the claim that $[F - D]$ is torsion.
\end{proof}

\begin{prop}[Observation~\ref{obs:intro-mm-girth}]
\label{prop:cycle-torsion}
Suppose $G$ has a cycle with exactly $d$ edges.
Then for any edge lengths $\ell : E(G) \to \RR_{>0}$, 
the metric graph $\Gamma = (G,\ell)$
fails to satisfy the degree~$d$ Manin--Mumford condition.
\end{prop}
\begin{proof}
Let $C$ be a cycle in $G$ 
with edges $e_1,e_2,\ldots,e_d$
and vertices $v_1,v_2,\ldots,v_d$
in cyclic order,
where edge $e_i$ has endpoints $v_i$ and $v_{i+1}$ 
(and indices are taken modulo $d$).
Consider the effective divisors
$D = v_1 + \cdots + v_d$
and 
$E = x_1 + \cdots + x_d$
where $x_i$ is the midpoint on edge $e_i$.

To show that $[D-E]$ is torsion,
we construct a piecewise linear function $f$ 
with $\Divisor(f) = D - E$.
Let $f:\Gamma \to \RR$ be zero-valued outside of the cycle $C$,
and $f(v_i) = 0$ for each vertex (potentially required  by continuity of $f$).
On each edge $e_i$, 
let $f$ have slope $\frac{1}{2}$ in the directions away from $v_i$,
so that at the midpoint $f(x_i) = \frac14 \ell(e_i)$.
It is straightforward to verify that $\Divisor(f) = D - E$
as desired.

The set of slopes of $f$, consisting of $\{0, 1/2, -1/2\}$,
is rational but not integral.
By Lemma \ref{lem:d-torsion-slope},
this implies that 
$[D-E]$ is a nonzero, torsion divisor class.
Moreover, both $[D]$ and $[E]$ lie in the same cell
$\Eff(e_1,\ldots,e_d)$.
Then Lemma~\ref{lem:infinite-torsion}
implies that the torsion packet $\tpacket{D}$
is infinite, which violates the degree $d$ Manin--Mumford condition.
\end{proof}

\subsection{Uniform Manin--Mumford bounds}
In this section,
we show that 
if a metric graph satisfies
the  Manin--Mumford condition,
then in fact
the number of torsion points can be bounded 
uniformly in terms of the genus of $\Gamma$.

\begin{thm}[Theorem~\ref{thm:intro-conditional-mm}]
\label{thm:mm-uniform-bound}
Suppose $\Gamma$ is a metric graph of genus $g \geq 2$.
If $ \Abelj{q}(\Gamma)\cap \Jtors{\Gamma}$ is finite,
then 
\[ 
	\#(\Abelj{q}(\Gamma)\cap \Jtors{\Gamma}) \leq 3g-3. \qedhere
\]
\end{thm}
\begin{proof}
The deformation retraction $r : \Gamma \to \Gamma'$
from a metric graph to its stabilization induces an isomorphism on Jacobians
$\Jac(\Gamma) \xrightarrow{\sim} \Jac(\Gamma')$
(Proposition~\ref{prop:jacobian-stabilization}) 
and hence on $\Abelj{q}(\Gamma) \xrightarrow{\sim} \Abelj{r(q)}(\Gamma')$,
so we may assume that $\Gamma$ is semistable
and that $ (G,\ell)$ is a stable combinatorial model for $\Gamma$.
Proposition \ref{prop:stable-edge-bound} 
states that $\# E(G) \leq 3g-3$ since $G$ is stable.
%
Lemma \ref{lem:infinite-torsion} (where $[D]$ and $[E]$ have degree one) implies that 
 a finite torsion packet has at most one point on a given edge of $G$.
This proves that the size of a finite, degree $1$ torsion packet is at most $3g-3$.
By Proposition~\ref{prop:mm-torsion-packet}, we are done.
\end{proof}

We next generalize the above argument to 
the higher-degree case.

\begin{thm}[Theorem~\ref{thm:intro-conditional-higher}]
\label{thm:mm-higher-uniform-bound}
Let $\Gamma = (G,\ell)$
be a connected metric graph of genus $g\geq 2$.
If $\Gamma$ satisfies the Manin--Mumford condition in degree $d$,
then
\[ \#(\Abeljh{D}{d}(\Gamma^d)\cap \Jtors{\Gamma}) \leq \binom{3g-3}{d} . \qedhere\]
\end{thm}
\begin{proof}
The number $\#(\Abeljh{D}{d}(\Gamma^d)\cap \Jtors{\Gamma})$
does not change under replacing $\Gamma$ with its stabilization,
so we may assume $\Gamma$ is semistable
and $(G,\ell)$ is a stable model.
This means that the number of edges $\# E(G)$ is bounded above by $3g-3$.

The image of $\Abeljh{D}{d}(\Gamma^d)$ is homeomorphic to $\Eff^d(\Gamma)$.
(They differ by a translation sending $\Pic^d(\Gamma)$ to $\Pic^0(\Gamma)$.)
By Theorem~\ref{thm:abks-deg-d},
the maximal cells of  $\Eff^d(\Gamma)$
are indexed by 
independent sets of size $d$ in the cographic matroid $M^\perp(G)$.
The number of maximal cells is clearly bounded above by
$\binom{\#E(G)}{d}$,
the number of all size-$d$ subsets of edges.
Since we assumed $G$ is stable, we have
$\binom{\#E(G)}{d}\leq \binom{3g-3}{d}$.

From Lemma~\ref{lem:infinite-torsion},
we know that a finite degree $d$ torsion packet contains 
at most one element from a given maximal cell of $\Abeljh{D}{d}(\Gamma^d)$,
which finishes the proof.
\end{proof}

\section{Manin--Mumford for generic edge lengths}
In this section we prove that metric graphs with very general edge lengths
satisfy the Manin--Mumford condition, in degrees $d = 1,2,\ldots, \gamind-1$ where $\gamind$ denotes the independent girth of the graph.

\subsection{Degree one}

We first address when a metric graph 
satisfies the Manin--Mumford condition in degree one.
In this section, ``torsion packet'' will always mean 
a degree-one torsion packet (c.f. Definition~\ref{dfn:torsion-packet}).
Before addressing the general case, we demonstrate an example in small genus.

\begin{eg}
Let $G$ be the theta graph (see Figure~\ref{fig:kirchhoff-ex1}) with vertices $x,y$ and
edges $e_1,e_2,e_3$,
and consider the metric graph $\Gamma = (G,\ell)$ with edge lengths 
$a = \ell(e_1), b = \ell(e_2) , c = \ell(e_3)$.

If a torsion packet contains two points on $e_1$,
then Proposition~\ref{prop:torsion-edge-deletion}
implies that $[x-y]$ is torsion on the deleted subgraph 
$\Gamma_1 = \Gamma \backslash e_1$.
By Lemma~\ref{lem:torsion-rational-slope},
this would imply
the potential function which sends current from $x$ to $y$ 
on the subgraph $\Gamma_1$
has rational slopes.
We can compute these slopes directly:
 $\Gamma_1$ is a parallel combination of edges with lengths $b$ and $c$,
so the slope along $e_2$ is
$\frac{c}{b+c} $.
(This calculation also follows from Theorem~\ref{thm:kirchhoff}.)
To summarize:
\[
	(\text{some torsion packet contains $\geq 2$ points of $e_1$})
	\qquad\text{implies}\qquad 
	\frac{c}{b+c} \in \QQ.
\]
The contrapositive statement is that
\[
	\frac{c}{b+c} \not\in \QQ
	\qquad\text{implies}\qquad 
	(\text{every torsion packet contains at most one point of $e_1$}).
\]

To satisfy the Manin--Mumford condition,
it suffices that every torsion packet $\tpacket{x} \subset \Eff^1(\Gamma)$
contains at most one point of each edge $e_1, e_2,e_3$.
Thus the Manin--Mumford condition holds for $\Gamma$
if the edge lengths are in the set
\[ 
	\left\{ (a,b,c) \in \RR_{>0}^3 : \frac{b}{a + b} \not\in \QQ 
\text{ and } \frac{c}{a + c} \not\in \QQ
\text{ and }\frac{c}{b+c} \not \in \QQ \right\} .
\]
This is a very general subset of $\RR_{>0}^3$, c.f. Example~\ref{eg:very-general} (b).
\end{eg}

\begin{prop}
\label{prop:torsion-edge-deletion}
Suppose $\Gamma$ is a metric graph and points $x,y\in \Gamma$ lie on the same non-bridge edge.
Let $\Gamma_0$ denote the metric graph with the open segment between $x$ and $y$ removed.
If $[x-y]$ is torsion on $\Gamma$,
then $[x-y]$ is torsion on $\Gamma_0$.
\end{prop}

%

\begin{proof}
Suppose $[x-y]$ is torsion on $\Gamma$.
Let $\potent{x}{y}$ denote the 
potential function on $\Gamma$ when one unit of current is sent from $y$ to $x$ (see Definition~\ref{dfn:unit-potential}).
By Lemma~\ref{lem:torsion-rational-slope},  all slopes of $\potent{x}{y}$ are rational.
In particular, the slope of $\potent{x}{y}$ on the segment between $x$ and $y$ is rational; 
let  $s$ denote this slope.
Let $\Gamma_0$ denote the metric graph obtained from $\Gamma$ by deleting the open segment between $x$ and $y$.
Then the restriction of $\potent{x}{y}$ to $\Gamma_0$ has Laplacian 
$\Divisor(\potent{x}{y}\big|_{\Gamma_0})= (1-s)y - {(1-s)}x$.
Since $\Gamma_0$ is connected, we have $s < 1$.

Let $\potent{x,0}{y}$ denote the potential function on $\Gamma_0$
when one unit of current is sent from $y$ to $x$.
Since $\potent{x}{y}\big|_{\Gamma_0} = (1-s)\potent{x,0}{y}$, 
all slopes of $\potent{x,0}{y}$ are rational.
By Lemma~\ref{lem:torsion-rational-slope}, this implies $[x-y]$ is torsion on $\Gamma_0$ as desired.
\end{proof}

\begin{prop}
\label{prop:slope-01}
Suppose $x,y$ are two vertices on a graph $G$.
Let $\potent{x}{y}$ be the potential function on $\Gamma = (G,\ell)$, 
depending on variable edge lengths $\ell : E(G) \to \RR$.
Either:

(1) all slopes of $\potent{x}{y}$ are $1$ or $0$, independent of edge lengths;
or

(2) for some edge $e$, the slope of $\potent{x}{y}$ along $e$ is a 
non-constant rational function of the edge lengths. 
\end{prop}
\begin{proof}
Suppose there is a unique simple path in $G$ from $x$ to $y$.
Let $f$ be the piecewise linear function on $\Gamma$
which has $f(x) = 0$,
increases with slope $1$ along the path from $x$ to $y$,
and has slope $0$ elsewhere.
Then  $f$ satisfies $\Divisor(f) = y - x$
so we must have $\potent{x}{y} = f$ by uniqueness.
Thus we are in case (1).

On the other hand, suppose there are two distinct simple paths $\pi_1, \pi_2$ 
in $G$ from $x$ to $y$. 
Let $e$ be an edge of $G$ which lies on $\pi_1$ but not $\pi_2$.
If we fix the lengths of edges in $\pi_1$ and send all other edge lengths to infinity,
then the slope of $\potent{x}{y}$ along $e$ approaches $1$.
If we send the length $\ell(e)$ to infinity while keeping all other edge lengths fixed, 
then the slope of $\potent{x}{y}$ along $e$ approaches zero.
Thus the slope of $\potent{x}{y}$ along $e$ is a non-constant function of the edge lengths.
By Kirchhoff's formulas, Theorem~\ref{thm:kirchhoff}, the slope 
is a rational polynomial function of the edge lengths.
This is case (2).
\end{proof}

\begin{prop}
\label{prop:vertex-torsion}
Suppose $x,y$ are two vertices on a graph $G$.
Then for the metric graph $\Gamma = (G,\ell)$, either

(1) $[x-y] = 0$ in $\Jac(\Gamma)$ for any edge lengths $\ell$,
or

(2) $[x-y]$ is non-torsion in $\Jac(\Gamma)$ for very general edge lengths $\ell$.
\end{prop}
\begin{proof}
If none of the slopes of $\potent{x}{y}$  vary as a function of edge lengths,
then by Proposition \ref{prop:slope-01} 
all slopes of $\potent{x}{y}$ are zero or one.
This implies that $[x-y]=0$.

On the other hand, 
suppose for some edge $e$
  the slope of $\potent{x}{y}$ along $e$ is a non-constant rational function 
$\frac{p(\ell_1,\ldots,\ell_m)}{q(\ell_1,\ldots,\ell_m)}$.
Then the subset
\[ 
U = \left\{ (\ell_1,\ldots, \ell_m) \in \RR_{>0}^m  : 
\frac{p(\ell_1,\ldots,\ell_m)}{q(\ell_1,\ldots,\ell_m)} \not\in \QQ \right\}
\]
parametrizing edge-lengths
where the slope at $e$ take irrational  values
is very general, c.f. Example~\ref{eg:very-general} (b).
By Lemma~\ref{lem:torsion-rational-slope}, 
$[x-y]$ is nontorsion on $U$, as desired.
\end{proof}

\begin{rmk}
In Propositions~\ref{prop:slope-01} and \ref{prop:vertex-torsion}, we have case (1) if and only if $x$ and $y$ are connected by a path of bridge edges.
\end{rmk}

\begin{thm}[Theorem~\ref{thm:intro-manin-mumford}]
\label{thm:manin-mumford}
Suppose $G$ is a biconnected metric graph of genus $g\geq 2$.
For a very general choice of edge lengths
$\ell : E(G) \to \RR_{>0}$,
the metric graph $\Gamma = (G,\ell)$
satisfies the Manin--Mumford condition.
\end{thm}
\begin{proof}

Let $m = \# E(G)$ and choose an ordering $E(G) = \{e_1, e_2, \ldots, e_m\}$,
which induces a homeomorphism 
from the space of edge lengths
$\{ \ell: E(G) \to \RR_{>0}\}$
to the positive orthant
 $\RR_{>0}^m$.
We claim that for each edge $e_i$,
there is a corresponding very general subset $U_i \subset \RR_{>0}^m$
such that 
\begin{equation}
\label{eq:torsion-ei-condition}
\begin{gathered}
\text{ when edge lengths are chosen in $U_i$,
every torsion packet } \\
\text{ of $\Gamma = (G,\ell)$
contains at most one point of $e_i$.}
\end{gathered}
\end{equation}

Let $e_{i}^+$, $e_{i}^-$ denote the endpoints of $e_i$,
and let $G_i = G \backslash e_i$ denote
the graph with the interior of $e_i$ deleted.
If the endpoints $e_{i}^+$, $e_{i}^-$
are not connected by any path in $G_i$,
this contradicts our assumption that $G$ is biconnected.
If the endpoints are connected by only one path $\pi$ in $G_i$,
then the union $\pi \cup \{e_i\}$ is a genus $1$ biconnected 
component of $G$,
which contradicts our assumption that $G$ is biconnected and has genus $g\geq 2$.
Thus $e_{i}^+$, $e_{i}^-$ are connected by
at least two distinct paths in $G_i$.
(The distinct paths may share some edges in common.)

Therefore, the divisor class 
$[e_{i}^+ - e_{i}^-] \neq 0$ in $\Jac(\Gamma_i)$
where $\Gamma_i = (G_i , \ell_i)$.
By Proposition \ref{prop:vertex-torsion},
$[e_{i}^+ - e_{i}^-] $ is nontorsion in $\Jac(\Gamma_i)$
on a very general subset 
$V_i \subset \RR_{>0}^{m-1}$
of edge-length space. (Note  that $G_i$ has $m-1$ edges.)
Finally, we let $U_i$ be the preimage of $V_i$ under the coordinate projection
$\RR_{>0}^m \to \RR_{>0}^{m-1}$ forgetting coordinate $i$.
The subset $U_i$ is very general,
and satisfies the claimed condition \eqref{eq:torsion-ei-condition}.

For any edge lengths in the 
intersection $U = \bigcap_{i=1}^m U_i$
a torsion packet of the corresponding $\Gamma = (G,\ell)$
can have at most one point on each edge $e_i$,
giving the bound
$ \# \tpacket{x} \leq m .$
The subset $U$ is very general, since it is a finite intersection of very general subsets.
This completes the proof.
\end{proof}

\subsection{Higher degree}

In this section we address when a metric graph with very general edge lengths 
satisfies the Manin--Mumford condition in higher degree ($d \geq 2$).

The next proposition is a strengthening of Proposition~\ref{prop:cycle-torsion}.
Recall that $M^\perp(G)$ denotes the cographic matroid of $G$
(see Section~\ref{subsec:matroids}).
\begin{prop}
\label{prop:gamind-cycle-torsion}
Suppose $G$ contains a cycle $C$
whose edge set has rank $d = {\rm rk}^\perp(E(C))$ 
in the cographic matroid $M^\perp(G)$.
Then for any edge lengths $\ell : E(G) \to \RR_{>0}$, 
the metric graph $\Gamma = (G,\ell)$
fails to fulfill the degree $d$ Manin--Mumford condition.
\end{prop}
\begin{proof}
Suppose the given cycle of $G$ consists of edges
$\{e_1, \ldots, e_k\}$ 
and vertices $\{v_1,\ldots, v_k\}$ in cyclic order;
note that $k\geq d$.
Let $D = v_1 + \cdots + v_k$
be the sum of the cycle's vertices.
In the proof of Proposition~\ref{prop:cycle-torsion},
we showed that the degree $k$ torsion packet $\tpacket{D}$ is infinite, 
since it has infinite intersection with
the cell $\Eff(e_1,\ldots, e_k)$,
for any choice of edge lengths $\ell$.
Our goal is to produce a torsion packet of degree $d$ which is infinite.

Recall that $\Eff(e_1,\ldots, e_k)$ is the image of 
$\Div(e_1,\ldots, e_k)$
under the 
linear equivalence map $\Div^k (\Gamma) \to \Pic^k(\Gamma)$.
By Theorem~\ref{thm:cographic-dim},
$\Eff(e_1,\ldots,e_k)$
has dimension $d = {\rm rk}^\perp(\{e_1,\ldots,e_k\})$.
This implies that 
$\Eff(e_1,\ldots,e_k)$ is covered by the images of the $d$-dimensional faces of $\Div(e_1,\ldots,e_k)$ 
under the linear equivalence map.
The $d$-dimensional faces of $\Div(e_1,\ldots,e_k)$ have the form
\begin{equation}
\label{eq:d-face}
 \Eff(e_i : i\in I) + \Big[\, \sum_{i \not\in I} v_{i}^{\pm} \,\Big] 
 \quad \subset\quad  \Eff(e_1,\ldots,e_k),
\end{equation}
where $I$ is a $d$-element subset of $\{1,\ldots,k\}$
and $v_i^{\pm} \in \{v_i, v_{i+1} \}$ is an endpoint of $e_i$.
(There are $\binom{k}{d} 2^{k-d}$ such choices.)
Since $\Eff(e_1,\ldots,e_k)$ has infinite intersection
with the torsion packet $\tpacket{D}$,
there is some choice of 
$I, v_i^\pm$
such that the subset 
$\Eff(e_i : i\in I) + [\, \sum_{i \not\in I} v_{i}^{\pm} \,]$
in \eqref{eq:d-face} has infinite intersection with $\tpacket{D}$.
This implies that 
$\tpacket{D - \sum_{i\not\in I} v_i^\pm}$
is a degree-$d$ torsion packet 
which has infinite intersection with $\Eff(e_i : i\in I)$,
thus violating the degree $d$ Manin--Mumford condition.
\end{proof}

Next, we consider 
the converse situation of Proposition~\ref{prop:gamind-cycle-torsion}, i.e. 
when an edge set is acyclic after taking the closure in $M^\perp(G)$.
Recall from Section~\ref{subsec:matroids}
the notation $\Div(e_1,\ldots,e_k)$ and $\Eff(e_1,\ldots,e_k)$.
Here we introduce a slight variation:
let $\Div(e_1,\ldots,e_k)^\circ$
denote the set of effective divisors
of the form $D = x_1 + \cdots + x_k$ where $x_i$
is in the {\em interior} $e_i^\circ$ of edge $e_i$;
respectively 
let $\Eff(e_1,\ldots,e_k)^\circ$ denote the divisor classes of the form $[x_1+ \cdots + x_k]$,
where $x_i \in e_i^\circ$.


\begin{prop}
\label{prop:acyclic-nontorsion}
Suppose $e_1,\ldots, e_k$
are edges in $G$
such that $\{e_1,\ldots,e_k\}$ is independent in $M^\perp(G)$
and
the closure of $\{e_1,\ldots,e_k\}$
in $M^\perp(G)$
spans an acyclic subgraph of $G$.
Then 
for very general edge lengths on $\Gamma = (G,\ell)$, 
distinct divisor classes in $\Eff(e_1,\ldots,e_k)^\circ \subset \Pic^k(\Gamma)$ are in distinct torsion packets.
\end{prop}

Before proving this statement, we introduce some lemmas and definitions.
\begin{definition}
\label{dfn:cv-active}
Given a piecewise linear function $f$  on $\Gamma = (G,\ell)$, 
say an edge of $G$ is {\em current-active} with respect to $f$
if the slope $f'$ is nonzero in a neighborhood of 
its endpoints\footnote{if
 $e \cong [0,1]$, here a ``neighborhood of the endpoints'' means $[0,\epsilon) \cup (1 - \epsilon, 1]$
for some $\epsilon>0$};
let $E^{\rm c.a.}(G,f)$ denote the current-active edges,
\begin{equation*}
E^{\rm c.a.}(G,f) = \{e \in E(G) : f'\neq 0 \text{ in a neighborhood of } e^+,e^- \text{ in }e\}.
\end{equation*}
Say an edge is {\em voltage-active} with respect to $f$ 
if the net change in $f$ across $e$ is nonzero;
let $E^{\rm v.a.}(G,f)$ denote the voltage-active edges,
\begin{equation*}
E^{\rm v.a.}(G,f) = \{e  \in E(G) : f(e^+) - f(e^-) \neq 0
\quad\text{where }e = (e^+,e^-) \}.
\qedhere
\end{equation*}
\end{definition}

Recall that a {\em cut} of $G$ is a set of edges $\{e_1,\ldots,e_k\}$
such that the deletion $G\setminus \{e_1,\ldots,e_k\}$ 
is disconnected.
\begin{lem}
\label{lem:v-active-cut}
Consider a metric graph 
$\Gamma = (G,\ell)$
and  $f \in \PL_\RR(\Gamma)$.
If $E^{\rm v.a.}(G,f)$ is nonempty,
it contains a cut of $G$.
\end{lem}
\begin{proof}
Suppose $e = (e^+,e^-)$ is voltage-active with respect to $f$,
so that $f(e^+) > f(e^-)$
for some ordering of endpoints.
Then we may partition $V(G)$ into two nonempty sets
$V^+ \cup V^-$,
where 
\[
V^+ = \{ v\in V(G) : f(v) \geq f(e^+) \}
\qquad\text{and}\qquad 
V^- = \{v \in V(G) : f(v) < f(e^+) \}.
\]
It is clear that $E^{\rm v.a.}(G,f)$ contains all edges
between $V^+$ and $V^-$;
such edges form a cut of $G$.
\end{proof}

\begin{lem}
\label{lem:c-active-cycle}
On $\Gamma = (G,\ell)$,
consider  $f \in \PL_\RR(\Gamma)$ such that
$\Divisor(f) = D - E$
for  $D,E \in \Div(e_1,\ldots,e_k)^\circ$.
If $E^{\rm c.a.}(G,f)$ is nonempty, then it
contains a cycle of $G$.
\end{lem}
\begin{proof}
Suppose $D = x_1 + \cdots + x_k$ and $E = y_1 + \cdots + y_k$
where $x_i, y_i \in e_i$.
Since the divisor $\Divisor(f)$ restricted to $e_i$ has the form $x_i - y_i$,
the slopes of $f$ along $e_i$ are as shown in Figure \ref{fig:edge-slopes},
where slopes are measured in the rightward direction.
\begin{figure}[h]
\centering
\includegraphics[scale=0.5]{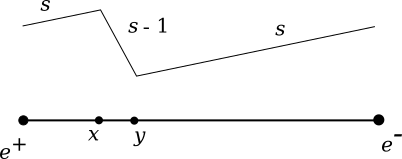}
\caption{Slopes on edge $e$ where $\Divisor(f) = x - y$.}
\label{fig:edge-slopes}
\end{figure}

\noindent
Edge $e_i$ is current-active
if and only if the corresponding slope $s$ ($ = s_i$) is nonzero. 
In particular, if $e_i \in E^{\rm c.a}(G,f)$ it is current-active {\em at both endpoints}.

On the other hand, consider an edge $e \in E(G) \setminus \{e_1,\ldots,e_k\}$.
Then $\Divisor(f)$ is not supported on $e$, so $f$ does not change slope on $e$.
Again in this case, if $e \in E^{\rm c.a.}(G,f)$ then 
it is current-active at both endpoints.

By assumption that divisors $D,E$ are in $\Div(e_1,\ldots,e_k)^\circ$,
the divisor $\Divisor(f) = D - E$ 
is supported away from the vertex set $V(G)$. 
This means that around a vertex $v$, the outward slopes of $f$ sum to zero.
The number of nonzero terms in the sum must be $0$ or $\geq 2$,
and each nonzero term corresponds to a current-active edge incident to $v$.
Thus
\begin{equation*}
\begin{gathered}
E^{\rm c.a.}(G,f)
\text{ spans a subgraph of $G$ where } \\
\text{ every vertex has } 
\val(v) = 0 \text{ or } \val(v) \geq 2.
\end{gathered}
\end{equation*}
The claim follows.
\end{proof}

\begin{lem}
\label{lem:cact-or-vact}
Consider  
$D,E \in \Div(e_1,\ldots,e_k)^\circ$
and $f \in \PL_\RR(\Gamma)$ such that
$\Divisor(f) = E - D$.
If $D \neq E$,
then $E^{\rm v.a.}(G,f)$ or $E^{\rm c.a.}(G,f)$
is nonempty (or both are).
\end{lem}
\begin{proof}
If $D = x_1 + \cdots + x_k$ is not equal to 
$E = y_1 + \cdots + y_k$,
then there is some index $i$ such that $x_i \neq y_i$.
Consider the illustration of $f$ in Figure~\ref{fig:edge-slopes},
applied to the edge with $x_i \neq y_i$.
We have
\begin{equation}
f(e^-_i) - f(e^+_i) = s \cdot \ell(e_i) - \ell([x_i,y_i]),
\end{equation}
where $\ell([x_i,y_i])$ is the distance between $x_i$ and $y_i$ on $e_i$.
If $s = 0$, then $e_i$ is not current-active but is voltage-active.
If $s = \ell([x_i,y_i]) / \ell(e_i)$, then $e_i$ is not voltage-active but is current-active.
\end{proof}

\begin{lem}
\label{lem:slope-nonconstant}
Consider a fixed vertex-supported $\RR$-divisor $D =  \lambda_1 v_1 + \cdots + \lambda_r v_r$ 
of degree zero on $G$,
so $v_i \in V(G)$, $\lambda_i \in \RR$ and $\sum_{i=1}^r \lambda_i = 0$.
On $\Gamma = (G,\ell)$, suppose $f \in \PL_\RR(\Gamma)$ satisfies $\Divisor(f) = D $
and  $f$ has nonzero slope on $e\in E(G)$. 
If $e$ is not a bridge,
then the slope on $e$ is a nonconstant rational function of edge lengths 
of $\Gamma$. 
\end{lem}
\begin{proof}
Suppose  
 we let $\ell(e) \to \infty$ for the chosen edge
 and fix the lengths of all other edges $e' \neq e$; 
we claim that the slope of $f$ across $e$ approaches zero.

The slope-current principle, Proposition~\ref{prop:slope-current},
states that the slope of $f$ 
is bounded above in magnitude by $\Lambda$,
where $\Lambda = \frac{1}{2}\sum_{i } |\lambda_i|$ 
does not depend on the edge lengths.\footnote{Since
$\sum \lambda_i = 0$, we have 
$\Lambda = \sum\{ \lambda_i : \lambda_i > 0 \}= -\sum \{\lambda_i : \lambda_i < 0\}$.}
Since $e = (e^+, e^-)$ is not a bridge edge, there is a simple path $\pi$ from $e^+$ to $e^-$ which does not contain $e$.
By integration along $\pi$,  $|f(e^-) - f(e^+)|$ is bounded above by
$\Lambda \cdot \ell(\pi)$,
which implies the bound
\begin{equation*}
 |f'(e)| = \left| \frac{f(e^-) - f(e^+)}{\ell(e)} \right| \leq \frac{\Lambda \cdot \ell(\pi)}{\ell(e)} .
\end{equation*}
If we let $\ell(e) \to \infty$ and keep $\ell(e')$ constant
for each $e' \in E(G) \setminus \{e \}$,
this upper bound approaches zero as claimed.

Thus the slope of $f$ along 
$e$ is a non-constant function of the edge lengths.
It is a rational function by Kirchhoff's formulas, Theorem \ref{thm:kirchhoff}.
\end{proof}

\begin{proof}[Proof of Proposition~\ref{prop:acyclic-nontorsion}]
Suppose $D = x_1 + \cdots + x_k$
and $E = y_1 + \cdots + y_k$
are divisors in $\Div(e_1,\ldots,e_k)^\circ$.
Let $f$ be a piecewise linear function such that $\Divisor(f) = E- D$.
By Lemma~\ref{lem:d-torsion-slope}, 
$[D]$ and $ [E]$ lie in the same torsion packet
if and only if 
all slopes of $f$ are rational.

Let $\Gamma_0$ (resp. $G_0$) denote the metric graph (resp. combinatorial graph) obtained from deleting 
the interiors of edges $e_1,\ldots, e_k$ from $\Gamma$ (resp. $G$).
Let  $f_0 = f\big|_{\Gamma_0}$ denote the restriction of $f$ to 
$\Gamma_0$.
We have
\begin{equation}
\label{eq:div-f0}
 \Divisor(f_0) = \lambda_1 w_1 + \cdots + \lambda_{r} w_{r},
\end{equation}
where $\{w_1,\ldots, w_r\}\subset V(G)$ is the set of endpoints of edges $e_1,\ldots, e_k$
and $\lambda_i \in \RR$.

First, suppose the tuple
$(\lambda_1,\ldots, \lambda_r) = (0,\ldots, 0)$.
Then $f_0$ is constant,
so every edge of $G_0$ is neither current-active nor voltage-active
with respect to $f$.
Since the edges $\{e_1,\ldots,e_k\}$ are assumed independent in $M^\perp(G)$,
they do not contain a cut of $G$
so the inclusion $E^{\rm v.a.}(G,f) \subset \{e_1,\ldots,e_k\}$
implies that ${ E^{\rm v.a.}(G,f) = \varnothing }$
by Lemma~\ref{lem:v-active-cut}.
Since the edges $\{e_1,\ldots,e_k\}$ do not contain a cycle of $G$,
the inclusion $E^{\rm c.a.}(G,f) \subset \{e_1,\ldots,e_k\}$
implies that $E^{\rm c.a.}(G,f) = \varnothing$
by Lemma~\ref{lem:c-active-cycle}.
Then Lemma~\ref{lem:cact-or-vact} implies that $D = E$.

Next, suppose the tuple
$(\lambda_1,\ldots, \lambda_r) \in \RR^r$ 
from \eqref{eq:div-f0} is nonzero.
This means that some edge of $G_0$
must be current-active,
so $E^{\rm c.a.}(G,f)$ is nonempty.
By Lemma~\ref{lem:c-active-cycle},
$E^{\rm c.a.}(G,f)$ contains a cycle of $G$.
The closure of $\{e_1,\ldots,e_k\}$ 
with respect to the cographic matroid $M^\perp(G)$
is equal to
\begin{equation*}
\{e_1,\ldots, e_k\} \cup \{b_1, \ldots, b_j\}
\quad\text{where $\{b_1, \ldots, b_j\}$
are the bridge edges of $G_0$.}
\end{equation*}
By assumption that
$\{e_1,\ldots, e_k\} \cup \{b_1, \ldots, b_j\}$
is acyclic,
$E^{\rm c.a.}(G,f)$
must contain an edge
$e_* \not\in \{e_1, \ldots, e_k\}$
which is not a bridge in $G_0$.
(The  edge $e_*$
depends on the tuple $(\lambda_1,\ldots, \lambda_r)$.)

Now consider applying Lemma~\ref{lem:slope-nonconstant} to the graph $G_0$,
the divisor \eqref{eq:div-f0}, and the edge $e_* \in E(G_0)$.
The lemma concludes that
as a function of the edge-lengths of $\Gamma_0$,
the slope of $f_0$ (equivalently $f$) on $e_*$
is a nonconstant ratio of polynomials.
In particular,
\begin{equation}
\label{eq:v-irrational}
V(\lambda_1,\ldots, \lambda_r) = 
\{ \text{edge lengths of $\Gamma_0$ such that $f_0'$ is irrational on $e_*$} \} 
\end{equation}
is a very general subset of $ \RR_{>0}^{m-k} \cong \{ \ell_0 : E(G_0) \to \RR_{>0} \}$,
and on this subset we have 
$[D]$ and $[E]$ are in distinct torsion packets.

Finally, let $U(\lambda_1,\ldots, \lambda_r)$ be the preimage of $V(\lambda_1,\ldots, \lambda_r)$
under the projection $\RR_{>0}^m \to \RR_{>0}^{m-k}$,
which is very general,
and let
\begin{equation*}
U = \bigcap_{\substack{(\lambda_1, \ldots, \lambda_r) \qquad \\ \in \QQ^r \setminus (0,\ldots, 0)}} 
U(\lambda_1,\ldots, \lambda_r)
\quad\subset\quad \RR_{>0}^m .
\end{equation*}
The subset $U$ is very general,
as a countable intersection of very general subsets.

If edge lengths of $\Gamma = (G,\ell)$ are chosen 
such that there are distinct divisors $D,E \in \Eff(e_1,\ldots, e_k)^\circ$
where $[D]$ and $[E]$ are in the same torsion packet,
then  the tuple
$(\lambda_1, \ldots, \lambda_r)$
as in \eqref{eq:div-f0}
must be rational and  nonzero.
Then the chosen edge lengths on $G_0 \subset G$
are excluded from the subset \eqref{eq:v-irrational},
hence the edge lengths 
are excluded also from $U$, as desired.
\end{proof}

\begin{thm}[Theorem~\ref{thm:intro-mm-higher-degree}]
\label{thm:mm-higher-degree}
Let $G$ be a connected graph of genus $g\geq 1$
and independent girth $\gamind$.
The metric graph $\Gamma = (G,\ell)$ 
satisfies the degree $d$ Manin--Mumford condition
for very general edge lengths $\ell: E(G) \to \RR_{>0}$
if and only if $1 \leq d < \gamind$.
\end{thm}
\begin{proof}
If $d\geq \gamind$, then 
$d \geq {\rm rk}^\perp(E(C))$
for some cycle $C$ of $G$.
Proposition~\ref{prop:gamind-cycle-torsion} states that 
$\Gamma$ fails
the Manin--Mumford condition in degree $d' = {\rm rk}^\perp(E(C))$,
so the condition also fails in degree $d\geq d'$.

Conversely if $d < \gamind$,
then for each $d$-subset of edges $ \{e_1,\ldots, e_d\}$,
its closure in $M^\perp(G)$ 
does not contain a cycle of $G$.
In particular, 
the edges for each maximal cell 
$\Eff(e_1,\ldots, e_d)$ 
of
$\Eff^d(\Gamma)$
satisfy the hypotheses of Proposition~\ref{prop:acyclic-nontorsion},
so there is a very general subset of edge lengths of $\Gamma$
for which every degree $d$ torsion packet
has at most one element in the chosen cell
$\Eff(e_1,\ldots, e_d)$.
Since there are finitely many maximal cells (cf. Theorem~\ref{thm:abks-deg-d}),
this implies that for very general edge lengths 
there are finitely many elements in each degree $d$ torsion packet.
\end{proof}

\begin{cor}
\label{cor:ind-girth-mm}
There exists a constant $M$ such that
for any metric graph  $\Gamma$ of genus $g \geq 1$,
if $\Gamma$ satisfies the Manin--Mumford condition in degree $d$,
then we have
\[ d < M \log g . \qedhere\]
\end{cor}
\begin{proof}
Let $\gamind = \gamind(\Gamma)$ denote the independent girth.
The result follows from Proposition~\ref{prop:gamind-cycle-torsion}, 
which implies that $d < \gamind$,
and the bound $\gamind  < M \log g$
from Corollary~\ref{cor:metric-girth-bound}.
\end{proof}

\section*{Acknowledgements}
The author gratefully thanks Sachi Hashimoto for the initial suggestion to study the tropical analogue of the Manin--Mumford conjecture, 
and David Speyer for suggesting the generalization to higher degree.
We thank Matt Baker for providing valuable discussion and several references to related work.
We thank the anonymous referees for many suggestions which improved this paper.
The material in this article initially appeared in the author's PhD dissertation.
This work was supported by the National Science Foundation [DMS-1600223, DMS-1701576]
and a Rackham Predoctoral Fellowship.

\bibliography{manin-mumford} 
\bibliographystyle{abbrv}

\end{document}